\documentclass[runningheads]{llncs}
\usepackage{booktabs}
\usepackage{graphicx}
\usepackage{url}
\usepackage{longtable}
\usepackage{geometry}
\usepackage{appendix}
\usepackage{bm}
\usepackage{amsmath}
\usepackage{fixltx2e}

\usepackage{amsfonts}
\usepackage{amssymb}
\usepackage{bbm}
\usepackage{bbold}
\usepackage{enumitem}

\usepackage{caption}
\usepackage{subcaption}
\usepackage[authormarkup=none, defaultcolor=blue]{changes}

\newcommand{\Real}[0]{\mathrm{\hspace{0.1mm}I\hspace{-0.8mm}R}}

\newcommand{\Natural}{\mathbbm{N}}
\newcommand{\kernel}{\mathrm{kern}}
\newcommand{\argmin}{\mathrm{argmin}}
\newcommand{\supp}{\mathrm{supp}}

\usepackage{array}
\newcolumntype{L}[1]{>{\raggedright\let\newline\\\arraybackslash\hspace{0pt}}m{#1}}
\newcolumntype{C}[1]{>{\centering\let\newline\\\arraybackslash\hspace{0pt}}m{#1}}
\newcolumntype{R}[1]{>{\raggedleft\let\newline\\\arraybackslash\hspace{0pt}}m{#1}}

\usepackage{algorithmic,algorithm,refcount}

\begin{document}   

\title{On the Number of Degenerate Simplex Pivots}
%
%
\author{Kirill Kukharenko\inst{1} \and
Laura Sanit\`a\inst{2}}
\authorrunning{K. Kukharenko and L. Sanit\`a}
%
\institute{Otto von Guericke University Magdeburg, Magdeburg, Germany \\
\email{kirill.kukharenko@ovgu.de}\and
Bocconi University, Milan, Italy\\
\email{laura.sanita@unibocconi.it}}
\maketitle   

\begin{abstract}
The simplex algorithm is one of the most popular algorithms to solve linear programs (LPs).
Starting at an extreme point solution of an LP, it performs a sequence of basis exchanges (called pivots) that 
allows one to move to a better extreme point along an improving edge-direction of the underlying polyhedron. 

A key issue in the simplex algorithm's performance is \emph{degeneracy}, which may lead to a 
(potentially long) sequence of basis exchanges which do not change the current extreme point solution.   
 
In this paper, we prove that 
\textcolor{black}{
one can
employ any improving feasible direction at an extreme point}
to limit the number of consecutive degenerate pivots that the simplex algorithm performs to $n-m-1$, where $n$ is the number of variables and $m$ is the number of equality constraints of a given LP in standard equality form.

\end{abstract}
\section{Introduction}

The simplex algorithm first introduced in \cite{dantzig1951} has always been one of the most widely used algorithms for solving linear programming problems (LPs). Although it performs as a linear
time algorithm for many real world linear programming models (see, e.g., the survey \cite{shamir1987}) and a variant of it (the shadow simplex) has been shown to have polynomial smoothed complexity (see, e.g., \cite{Borgwardt87,Borgwardt99,dadush2020,Megiddo1986,spielman04,vershynin09}), no polynomial version of the simplex algorithm has been found until these days, despite decades of research. 


The simplex algorithm relies on the concept of \emph{basis}, where a basis corresponds to an inclusion-wise minimal set of tight constraints which defines an extreme point solution of the underlying linear program. In each step, it performs a basis exchange by replacing one constraint currently in the basis with
a different one \textcolor{black}{that is not in the basis}.  Such an exchange is called a \emph{pivot}. From a geometric perspective, a basis exchange identifies a direction to move from the current extreme point. The simplex algorithm only considers pivots that yield an \emph{improving} direction, with respect to the objective function to be optimized.
As the result of pivoting, it is possible for the algorithm to either stay at the same extreme point, or to move to an adjacent extreme point of the underlying polyhedron. We refer to a pivot of the former type as \emph{degenerate}, in contrast to a pivot of the latter type, which will be called \emph{non-degenerate}. The \emph{pivot rule} which determines how to perform these basis exchanges (i.e., deciding which inequality enters and which inequality leaves the basis) is the core of the simplex algorithm. Various pivot rules have been introduced over decades, but none of them yields a polynomial time version of the simplex algorithm so far~(see, e.g., \cite{AvisFriedmann2017,Disser-Hopp19,Goldfarb94,GOLDFARB+sit1979,Hansen+Zwick2015,terlaky+zhang,kleeminty,zadeh2} and the references therein). 

When analyzing the performance of the simplex algorithm, one can often assume that the underlying polyhedron $\mathcal P$ is non degenerate, e.g., by performing a small perturbation of the constraint's right-hand-side. In this case, a necessary condition to guarantee an efficient performance of the simplex algorithm, is that there is always a way to reach an optimal solution within a polynomial number of non-degenerate pivots. Geometrically, this would imply that there is always a path between two extreme points on the 1-skeleton of $\mathcal P$ (the graph where vertices correspond to the extreme points of $\mathcal P$, and edges correspond to the 1-dimensional faces of $\mathcal P$), whose length is bounded by a polynomial in the input description of $\mathcal P$. This is a major and still open question in the field, \textcolor{black}{commonly referred to as the polynomial Hirsch Conjecture} (see for example \cite{disumma2014,dadush16,@matschke2015,santos2010,todd2014}).

On the other hand, even in the case when a short path exists, it is not clear how the simplex algorithm would be bound to follow it (see for instance the case of $0/1$-polytopes discussed in~\cite{black2021}, \textcolor{black}{or the hardness results in~\cite{DKS22,CarSte}}). In fact, given any LP one can apply some transformation to it to ensure the existence of a short path between an initial extreme point solution and an optimal one in the underlying polyhedron (such as, e.g., relying on an extended formulation with one extra dimension to construct a pyramid whose apex is adjacent to all extreme points of the initial polyhedron). Such standard transformations usually make the polyhedron become highly degenerate. 
In the presence of degeneracy, an extreme point might have exponentially many distinct bases defining it, and the simplex algorithm might perform an exponential number of consecutive degenerate pivots, staying in the same vertex, (see, e.g., \cite{cunningham79}). Such behavior is referred to as \emph{stalling}. In some pathological cases, certain pivot rules might even provide an infinite sequence of consecutive degenerate pivots at a degenerate vertex, a phenomenon called \emph{cycling}. Although cycling can be easily avoided by employing, for instance, Bland's rule \cite{bland77} or a lexicographic rule~\cite{Terlaky2009}, there is no known pivot rule that prevents stalling. As stated in several papers (e.g.,~\cite{Megiddo86degen,Murty2009}), it is well known that solving a general LP in strongly polynomial time can be reduced to finding a pivot rule that prevents stalling in general polyhedra, \textcolor{black}{and that can be implemented in strongly polynomial time}.
However, there are a few cases in which it is known that the issue of stalling can be handled. The most famous example is the class of transportation polytopes, for which pivot rules with polynomial bounds on the number of consecutive degenerate pivots were introduced in~\cite{cunningham79} and further developed in \cite{orlin02,goldfarb90,Rooley-Laleh81}. See also the work of~\cite{Orlin97} for a strongly polynomial version of the primal network simplex. 

Our findings are inspired by~\cite{Kabadi2008} to a great extent. The latter work shows that, given an LP in standard equality form with a totally unimodular coefficient matrix $A \in \mathbb \{-1,0,1\}^{m \times n}$, a vertex solution $x$ and an improving feasible \emph{circuit direction} for $x$, one can construct a pivot rule which performs at most $m$ consecutive degenerate pivots. We generalize their proof, and show that one can employ \emph{any} improving feasible direction at a vertex $x$ of a \emph{general} polyhedron, to avoid stalling.  Our result, proved in Section~\ref{sec:antistalling}, is summarized in the following theorem.

\begin{theorem}\label{th:pivot-rule}
    Given any LP of the form $\min\{c^T x: Ax=b, x\geq 0\}$ with $A \in \Real^{m \times n}$, there exists a pivot rule that limits the number of consecutive degenerate simplex pivots at any non-optimal extreme solution to at most $n-m-1$. 
    
    \textcolor{black}{The pivot rule can be efficiently implemented whenever an improving feasible direction at a non-optimal extreme point is available.}
\end{theorem}

We stress here that the pivots considered in Theorem~\ref{th:pivot-rule} are degenerate \emph{simplex} pivots, meaning that each of them yields an \emph{improving} direction at the given extreme point, though this direction is not feasible. It is important to point this out as in general, given two adjacent extreme points $x, x'$, one can easily perform a sequence of basis exchanges that yields $x'$ from $x$ by identifying the common tight linearly independent constraints, and introducing each of them to the current basis in any order, until the direction $x'-x$ is seen. However, we here want a strategy that guarantees that \emph{each} basis exchange \textcolor{black}{defines an improving direction, and hence is realizable by the simplex algorithm.}    

We then discuss some byproducts of our result in Section~\ref{sec:applications}. In particular, we revise the analysis of the simplex algorithm by Kitahara and Mizuno \cite{kitahara2013,kitahara2014} who bound the number of non-degenerate pivots in terms of $n,m$ and the maximum and the minimum non-zero coordinate of a basic feasible solution. Their analysis combined with our degeneracy-escaping technique show that the \emph{total} number of simplex pivots (both degenerate and non-degenerate) can be bounded in a similar way. 
As a consequence, one can have a strongly-polynomial number of simplex pivots for several combinatorial LPs. 
In addition, we perform some computational experiments to
evaluate the performance of the antistalling pivot rule in practice, reported in Section~\ref{sec:comp}.

Of course the drawback of the whole machinery is that it requires an improving feasible direction at a given vertex. Though it is efficiently computable, this is in general as hard as solving an LP. 
For some classes of polytopes though (e.g., matching or flow polytopes) finding such a direction can be easier, thus making it worthwhile to apply our pivot rule. Most importantly, we think that the main importance of our result is from a theoretical perspective: it shows that, for several polytopes, not only a short path on the 1-skeleton exists, but a short sequence of \emph{simplex pivots} always exists (and can also be efficiently computed). That is, a short sequence of \emph{basis exchanges} that follow \emph{improving} directions, when performing \emph{both} degenerate and non-degenerate pivots.

\section{Antistalling pivot rule}
\label{sec:antistalling}

The goal of this section is to prove Theorem~\ref{th:pivot-rule}. Before doing that, we state some preliminaries, and also give 
a geometric intuition behind the result.

\subsection{Preliminaries}
The simplex algorithm works with LPs in standard equality form:

\begin{equation}\label{eq:LP}
   \begin{array}{rll}
       \min & c^Tx & \\
        \mbox{s.t.} & Ax  &= b\\
        & x &\ge 0
   \end{array}
\end{equation}

For $A \in \Real^{m \times n}$ and $B \subseteq [n] \textcolor{black}{:=\{1,2, \dots, n\}}$, we use $A_B$ to denote the submatrix of $A$ formed by the columns indexed by $B$. 
Similarly, for $x\in \Real^n$ and $B\in [n]$, the notation $x_B$ is used to denote a vector consisting of the  entries of $x$ indexed by $B$. 
A \emph{basis} of \eqref{eq:LP} is a subset $B \subseteq [n]$ with $|B| = m$ and $A_B$ being non-singular. 
The point $x$ with $x_B = A_B^{-1}b$, $x_{N} = 0$ where $N:=[n]\setminus B$ is a \emph{basic solution} of \eqref{eq:LP} with basis $B$. If additionally $x_B \ge 0$, both the basic solution $x$ and the basis $B$ are \emph{feasible}. If $x_i > 0$ for all $i \in B$, then $B$ and $x$ are called \emph{non-degenerate}. \textcolor{black}{If instead $x_i=0$ for some $i \in B$, then $B$ and $x$ are called \emph{degenerate}}. We let $\overline{A} := A_B^{-1}A_N$ and $\overline{c}_N^T := c_N^T - c_B^T \overline{A}$. In particular, $\overline{c}_N\in \Real^N$ is the vector of so-called \emph{reduced costs} for the basis $B$.  
The coordinates of the reduced cost vector will be addressed as $\overline{c}_{N,i}$, where the first subscript will be dropped if the basis $B = [n] \setminus N$ is clear from the context. 

The simplex algorithm considers in each iteration a feasible $B$. If all elements of $\overline{c}_N$ are non-negative, then the basis $B$ and the corresponding basic feasible solution $x$ are known to be optimal. Otherwise, the algorithm \emph{pivots} by choosing a non-basic coordinate with negative reduced cost to enter the basis, say $f$. It then performs a minimum ratio test to compute an index $i$ that minimizes $\frac{x_i}{\overline{A}_{if}}$
among all indices $i \in B$ for which $\overline{A}_{if} >0$. Such index $i$ will be the  one leaving the basis, and it corresponds
to a basic coordinate which hits its non-negativity bound first when moving along the direction given by the tight constraints indexed by $B\setminus \{f\}$. 
At each iteration there could be multiple candidate indices for entering the basis (all the ones with negative reduced cost), as well as multiple candidate indices to leave the basis (all the ones for which the minimum ratio test value is attained). The choice of the entering and the leaving coordinates is the essence of a pivot rule (see~\cite{terlaky+zhang} for
a survey). A pivot is \emph{degenerate} if the attained minimum ratio test value is 0 (hence, the extreme point solution $x$ does not change), and \emph{non-degenerate} otherwise. Note that only coordinates with negative reduced cost are considered for entering the basis, since only such a choice guarantees the simplex algorithm to make progress when a non-degenerate pivot occurs. To later emphasize that we only refer to that kind of \emph{improving} pivots, we will call them \emph{simplex pivots}.

Finally, we let $\kernel(A):= \{y \in \Real^n: A y =0\}$. Given~\eqref{eq:LP}, we call a vector $y \in \kernel(A)$ with $c^Ty < 0$ an \emph{improving direction}. Such $y$ is said to be \emph{feasible} at a basic feasible solution $x$ if $x+\varepsilon y \ge 0$ for sufficiently small positive $\varepsilon$. Note that for any $y \in \kernel(A)$ and any basis $B$ of $A$, the following holds:
\begin{equation}\label{eq:basis_decomp}
    c^Ty= c^T_By_B + c^T_Ny_N = c^T_B(-\overline{A}y_N) + c^T_Ny_N = \overline{c}_N^Ty_N
\end{equation}


\subsection{A geometric intuition}

Before providing a formal proof, we give a geometric intuition on how our pivot rule works. For this, it will be easier to abandon the standard equality form. In particular, consider a degenerate vertex $x \in \Real^d$ of a polytope $\mathcal P$ of dimension $d \in \Natural$ defined by inequalities $a_i^Tx \le b_i$ with $a_i \in \Real^d, b_i \in \Real, i \in [n]$. In this setting, degeneracy means that more than $d$ inequalities are tight at $x$. 
Consider a subset $N$ of the set of indices of all inequalities that are tight at $x$ with $|N| = d$ and let $B:=[n]\setminus N$ \textcolor{black}{(see Figure~\ref{fig:antistalling})}. Note that here we intentionally redefine a notation used in the standard equality form, where non basic coordinates always have their corresponding constraints tight.
Since $x$ is degenerate, there is at least one inequality that is tight at $x$ with index in $B$. We denote the subset of such inequalities by $W \subseteq B$. Finally, assume $x$ is not an optimal vertex of $\mathcal P$ when minimizing an objective function $c \in \Real^d$ over $\mathcal P$, and let $y^0 \in \Real^d$ be an improving feasible direction at $x$, i.e., such that $x+\varepsilon y^0 \in \mathcal P$ for a sufficiently small positive $\varepsilon$ and $c^Ty^0 < 0$. 
In order to find an improving edge of $\mathcal P$ at $x$, we look at the directions given by the extremal rays of the basic cone $C(N) := \{x\in \Real^n \mid a_i^T x \le 0, i \in N\}$. If there is an improving feasible direction among them, we are done. 
Otherwise, we reduce the dimension of the polytope in the following way. We pick an extremal ray $z \in C(N)$ such that $c^Tz < 0$. Note that $z$ is formed by $d-1$ inequalities from $N$. Let $f \in N$ be the only inequality index not used to define $z$\textcolor{black}{, among the $d$ tight inequalities in $C(N)$ ($f$ is depicted as the gray fading facet in Figure~\ref{fig:antistalling})}. Consider a vector combination $y^0 + \alpha z$ where $\alpha \ge 0$ and note that it provides an improving direction for any positive $\alpha$. Since $y^0$ is contained in the feasible cone $C(N\cup W) = \{x\in \Real^n \mid a_i^T x \le 0, i \in N \cup W\}$ but $z$ is not, the vector $y^0 + \alpha z$ leaves $C(N\cup W)$ for sufficiently large $\alpha$. Hence there has to be a number $\alpha^1 \ge 0$ such that $y^0 + \alpha^1 z$ is contained in a facet of $C(N\cup W)$. Without loss of generality assume that the latter facet is defined by the $g\textsuperscript{th}$ inequality \textcolor{black}{(green facet in Figure~\ref{fig:antistalling})}. Note that $g\in W$ since $y^0$ and $z$ are both contained in the basic cone $C(N)$ and so is $y^0 + \alpha z$ for any non-negative $\alpha$.
Then, we define a new improving feasible direction $y^1 := y^0 + \alpha^1 z$, $N = \big(N\setminus\{f\}\big)\cup\{g\}$, $B = \big(B \setminus\{g\}\big) \cup \{f\}$ and continue searching for an improving edge inside the facet defined by the $g\textsuperscript{th}$ inequality. Since $x+\varepsilon y^1$ belongs to this facet for an $\varepsilon >0$ small enough, the latter yields dimension reduction. After at most $d-1$ such steps we are bound to encounter a one-dimensional face of $\mathcal P$ containing $x+\varepsilon y^{\prime}$, where $y^{\prime}$ is an improving feasible direction at $x$. 
For an illustration see Figure \ref{fig:antistalling}.
\begin{figure}[h]
    \centering
    \begin{subfigure}[b]{0.47\textwidth}
        \includegraphics[width=\textwidth]{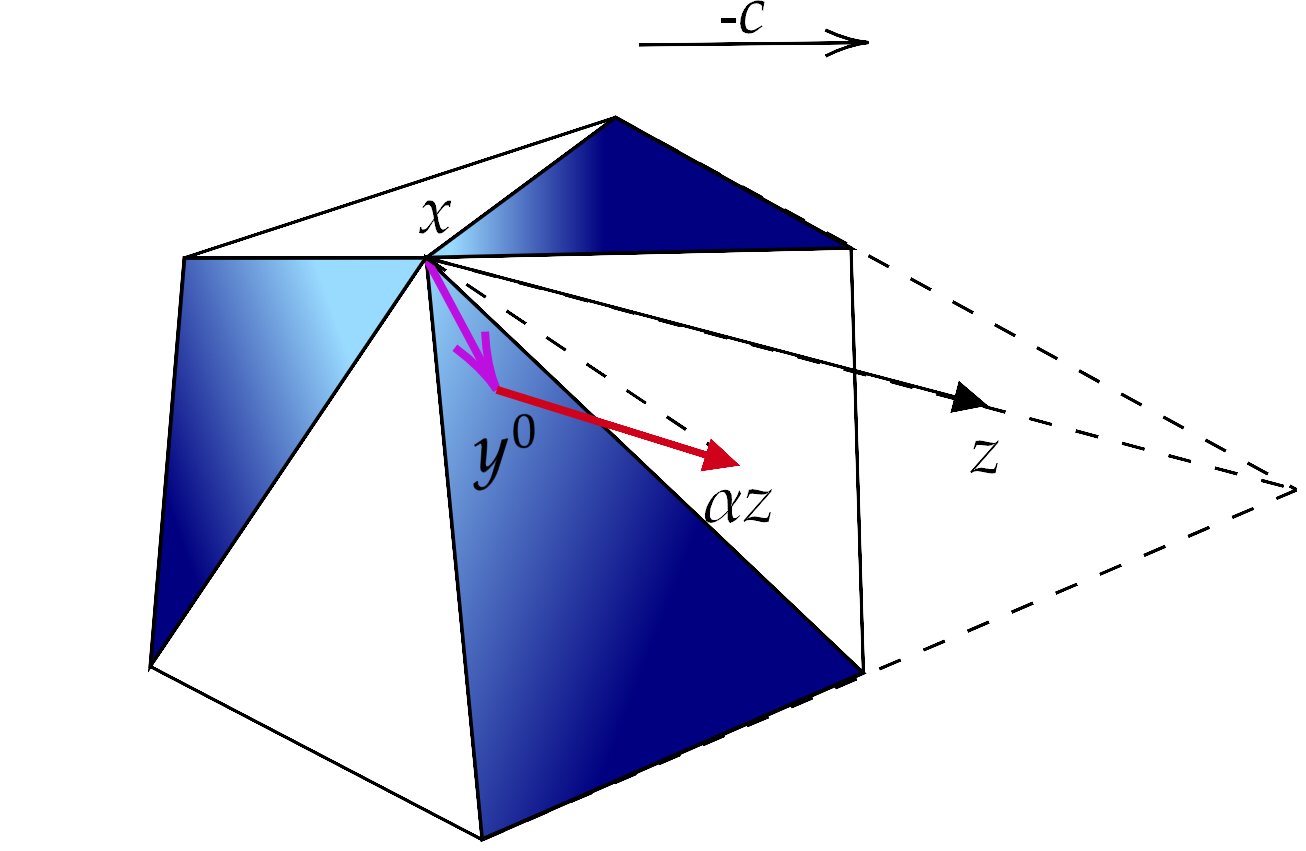}
    \end{subfigure}
    \hfill
    \begin{subfigure}[b]{0.47\textwidth}
        \includegraphics[width=\textwidth]{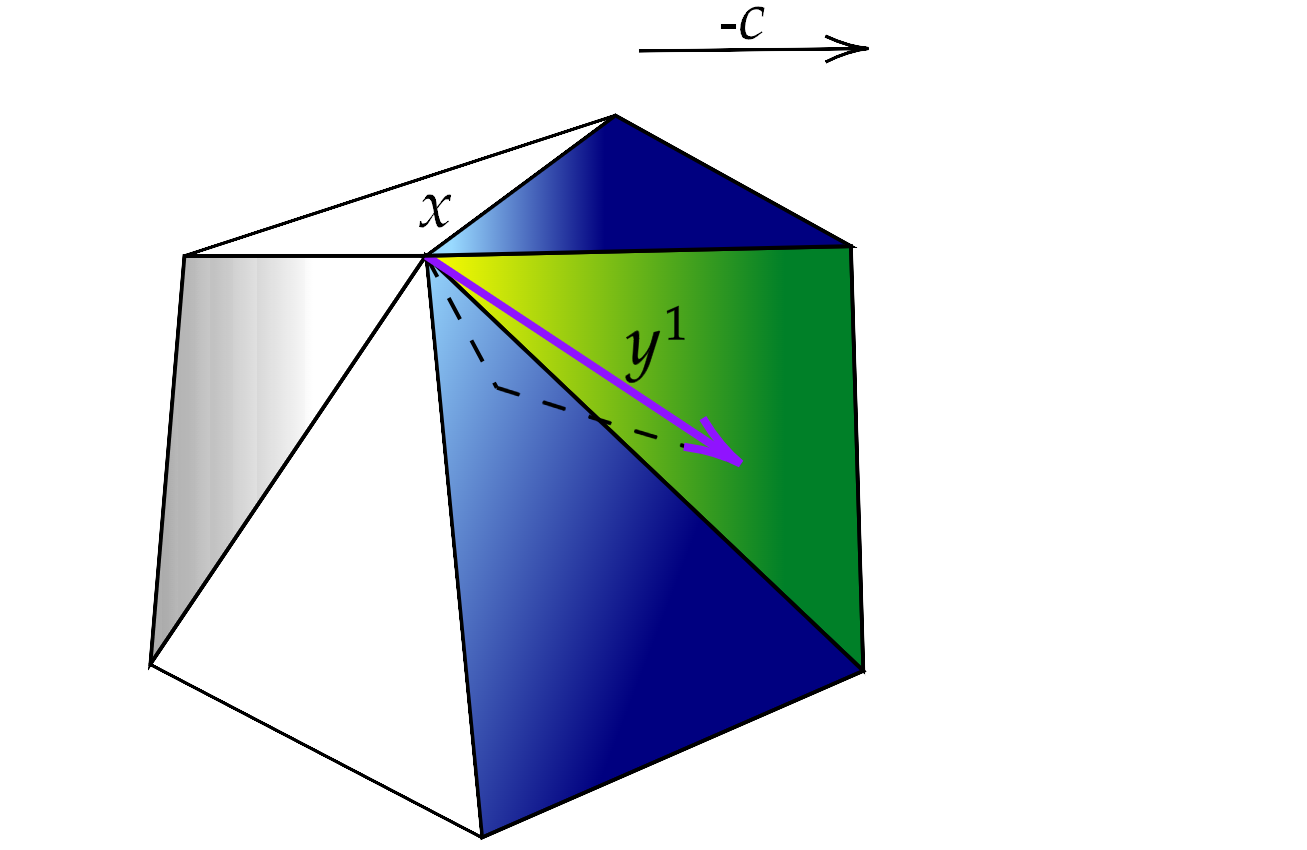}
    \end{subfigure}
    \caption{To the left, facets corresponding to inequalities in $N$ and $W$ are colored in navy and white, respectively. To the right, the gray fading facet is defined by the $f\textsuperscript{th}$ inequality and the green one corresponds to the $g\textsuperscript{th}$.}
    \label{fig:antistalling}
\end{figure}

\subsection{Proof of Theorem \ref{th:pivot-rule}}
\begin{proof}

Let $x$ be a degenerate basic feasible solution of \eqref{eq:LP} with basis $B$. Without loss of generality, assume $B = [m]$. Then, $x_i=0$ for $i \in N = [n]\setminus[m]$. Assume also that $x_i = 0$ for $i\in[k]$ with $1 \le k \le m$ and $x_i>0$ for $i=k+1,\dots,m$. Let $S(B):= \{i \in N \mid \overline{c}_i < 0\}$. Since $x$ is not optimal, $S(B) \neq \emptyset$.

It is well known that a basic feasible solution $x$ of \eqref{eq:LP} with basis $B$ is not optimal if and only if there exists an improving feasible direction, i.e., there exists $y^0 \in \Real^n$ such that $Ay^0= 0$, $c^Ty^0 < 0$ and $y^0_i \ge 0$ for all $i \in [k] \cup \big([n]\setminus[m]\big)$. 
Consider any such $y^0$. Let $Q_1(y^0, B):=\{i\in[k]\mid y^0_i > 0\}$ and $Q_2(y^0, B):=\{i\in[n]\setminus[m] \mid y^0_i > 0\}$. Without loss of generality, let $Q_1(y^0, B) = [r]$ with $0 \le r \le k$ and $Q_2(y^0, B) = \{m+1,\dots,m+t\}$ with $1\le t \le n-m$. See Table \ref{tab:pivotrule} for an illustration.

\renewcommand{\arraystretch}{1.5}
\begin{table}[]
    \centering
    \begin{tabular}{c|cccc|ccc|ccc|cccc|ccc}
         & 1&  & \dots & $r$ & & \dots &$k$ & & \dots & $m$ & $f$ & & \dots & $m+t$ & & \dots & $n$ \\
         \hline
    $x$ & 0 & 0 & \dots & 0 & 0 & \dots & 0 & + &\dots & + & 0 &0 & \dots & 0  & 0 & \dots & 0 \\
       
     $y^0$ & +& + &\dots& + & 0& \dots& 0 & $\star$ &\dots& $\star$ & + & + &\dots& +& 0 &\dots& 0 \\

     $z$ & -- &$\star$ &\dots& $\star$  & $0_+$ &\dots& $0_+$ & $\star$& \dots &$\star$ & 1 &0& \dots& 0 & 0 &\dots &0 \\

      $y^0 +\alpha z$ & 0 &$0_+$ &\dots &$0_+$  & $0_+$ &\dots &$0_+$ & $\star$ &\dots& $\star$ & + &+&\dots& + & 0 &\dots& 0 
      
    \end{tabular}
    
    \caption{An illustration of entries of $x, y^0, z$ and $y^1:=y^0+\alpha z$. A positive, a non-negative, a negative,  and a sign-arbitrary entry is denoted by $+$, $0_+$, $-$ and $\star$, respectively. Without loss of generality, we here assumed that the entering index $f$ is $m+1$, and that an index $g$ for which the minimum in \textbf{Case III} is attained is 1.}
    \label{tab:pivotrule}
\end{table}

Since $c^Ty^0 = \overline{c}_N^T y_N^0 < 0$ due to \eqref{eq:basis_decomp}, it follows that $S(B) \cap Q_2(y^0, B) \neq \emptyset$. We choose the entering index to be a non-basic one with the most negative reduced cost\footnote{We would like to point out that to prove the theorem, it suffices to choose \emph{any} non-basic index with negative reduced cost in the support of $y^0$ as the entering index. The fact that $f$ has the most negative reduced cost among such indices will only be used for the results in the next section.} in the support of $y^0$, that is $f = \argmin_{i \in S(B) \cap Q_2(y^0, B)} \overline{c}_i$. To detect the leaving index, we consider the following case distinction. 
Note that the case distinction depends on the basis $B$ and the improving feasible direction $y^0$.

\smallskip
\textbf{Case I:} \emph{There is no $i\in [k]$ with $\overline{A}_{if} > 0$}. 
In this case, the minimum ratio test for $f$ is strictly positive, that is: 
$$(*) \quad \min_{i \in Eq^>(f)} \frac{x_i}{\overline{A}_{if}} >0 \mbox{ where } Eq^>(f) := \{i\in [m] \mid \overline{A}_{if} > 0\}\,.$$ 
In particular, this means $Eq^>(f) \subseteq [m]\setminus[k]$. We perform a non-degenerate pivot, by selecting an index that minimizes $(*)$ as the leaving index. 

\smallskip
\textbf{Case II:} \emph{$\overline{A}_{if} > 0$ for some $i \in [k]\setminus[r]$}.  
In this case, we perform a degenerate pivot by selecting $i$ as the leaving index. Let $B^\prime := B \cup \{f\} \setminus\{i\}$. Because of degeneracy, the basic solution associated with $B'$ is still $x$, and hence $y^0$ is still an improving feasible direction for $x$. Note that $|Q_2(y^0, B^\prime)| = |Q_2(y^0, B)| - 1$ since $y^0_f > 0$ but $y^0_i = 0$ by definition. Repeat the same for $B^\prime$ and $y^0$.

\smallskip
\textbf{Case III:} \emph{$\overline{A}_{if} \le 0$ for all $i \in [k]\setminus[r]$ and $\overline{A}_{if} > 0$ for some $i \in [r]$}. 
In this case, we are going to change our improving feasible direction. Consider the following vector $z \in \Real^n$ which is, in fact, an improving (though not feasible) direction for the entering variable $x_f$: 
\begin{equation}\label{eq:ray}
   z_i=\begin{cases}
    -\overline{A}_{if} & \text{for $i\in [m]$,}\\
    1 & \text{for $i=f$,}\\
    0 & \text{otherwise.}
  \end{cases}
\end{equation}
Note that $Az = 0$ and $c^Tz = \overline{c}_f < 0$. Moreover $z_i \ge 0$ for each $i \in \big([r]\setminus Eq^>(f)\big) \cup \big([k]\setminus[r]\big) \cup \big([n]\setminus[m]\big)$ and $z_i <0$ for $i \in Eq^>(f)$ by definition. Set 
$$y^1:=y^0 + \alpha z \; \mbox{ with } \; \alpha := \min_{i \in Eq^>(f) \cap [r]} \frac{y_i^0}{-z_i} > 0\,.$$ 
Observe that $y^1$ is a feasible direction for $x$, since $A(y + \alpha z) = 0$, and $y^1_i \ge 0$ for 
$i \in [k]\cup\big([n]\setminus[m]\big)$, because of the choice of $z$ and $\alpha$. Furthermore, $y^1$ is an improving direction, because $c^Tz < 0$ and hence 
\begin{equation}\label{eq:next_sec}
c^Ty^1 = c^Ty^0 + \alpha c^Tz < c^Ty^0 < 0\,.
\end{equation}
Note that $Q_2(y^1, B) = Q_2(y^0, B)$.
Furthermore, let $g \in Eq^>(f) \cap [r]$ be the index for which the value of $\alpha$ is attained. We have $y^1_g = 0$ and $\overline{A}_{gf}>0$. 
We now repeat the case distinction above for $B$ and $y^1$.

The key observation is that when the third case of the above case distinction has occurred, repeating the same for $B$ and $y^1$  falls into the second case (because the basis $B$ has not changed, meanwhile $y^1_g = x_g = 0$ with $g \in B$ and $\overline{A}_{gf}>0$). Therefore, after each degenerate pivot, the cardinality of the support of the improving feasible direction in the non-basic indices decreases. Hence, a sequence of $|Q_2(y^0,B)|$ degenerate pivots would yield an improving feasible direction $y'$ and a basis $B'$ of $x$ with $Q_2(y', B') = 0$, which in turn implies $0 = \overline{c}^T_{N'}y'_{N'} = c^Ty'$ due to \eqref{eq:basis_decomp}, yielding a contradiction. Hence the number of consecutive degenerate pivots with the suggested pivot rule cannot exceed  $n-m-1$. \qed

\end{proof}

For the sake of clarity, the antistalling pivot rule is formally summarized in Algorithm \ref{alg:pivot}.
\begin{algorithm}[h] 

\caption[Antistalling pivot rule]{Antistalling pivot rule.} 
\label{alg:pivot} 
\begin{algorithmic}[1]
 \REQUIRE A basic feasible solution $x$ of $\eqref{eq:LP}$ with basis $B = \{ 1,\dots,m\}$ and an improving feasible direction $y$. 
 \vspace{1em}

 \STATE Compute the matrix $\overline{A} := A_B^{-1}A_N$ and the reduced cost vector $\overline{c}_N^T := c_N^T - c_B^T \overline{A}$, where $N:=[n]\setminus B$.
 \STATE Compute the entering index $f$ as $\argmin_{\{i\in N:\overline{c}_i<0\} \cap \supp(y)} \overline{c}_i$.
 \IF{ $\overline{A}_{if} \le 0$}
    \STATE The optimal cost is unbounded: the algorithm terminates.
   \ELSE
    \STATE Perform the minimum ratio test: \newline $r := \min_{i \in \supp\big(\overline{A}_{if}\big)} \frac{x_i}{\overline{A}_{if}}$ and $J := \argmin_{i \in \supp\big(\overline{A}_{if}\big)}\frac{x_i}{\overline{A}_{if}}$. 

    \IF{$r>0$ (\textbf{case I})}
        \STATE Choose $j \in J$ as the leaving variable. Perform non-degenerate pivot: \\$B=B\cup\{f\}\setminus\{j\}$. 
    \ELSIF{$J \cap B \setminus \supp(y) \neq \emptyset$ (\textbf{case II})}
        \STATE Choose $j \in J \cap B \setminus \supp(y)$ as the leaving variable. 
                Perform degenerate pivot: $B=B\cup\{f\}\setminus\{j\}$. Go back to 1.
    \ELSE
        \STATE (\textbf{case III}) Compute improving (though not feasible) direction $z$ as: 
\begin{equation*}
   z_i=\begin{cases}
    -\overline{A}_{if} & \text{for $i\in B$,}\\
    1 & \text{for $i=f$,}\\
    0 & \text{otherwise.}
  \end{cases}
\end{equation*}
Compute the new improving feasible direction $y := y+\alpha z$, where $\alpha := \min_{i \in \supp(\overline{A}_{if}) \cap \supp(y) \setminus \supp(x)} \frac{y_i}{\overline{A}_{if}}$. Go back to 3.
    \ENDIF
\ENDIF
\end{algorithmic}
\end{algorithm}

We conclude with a \textcolor{black}{strengthening of the bound of Theorem~\ref{th:pivot-rule}} which will be useful in the next section. 

\begin{theorem}\label{th:pivot-rule2}
    \textcolor{black}{
    Given any LP of the form $\min\{c^T x: Ax=b, x\geq 0\}$ with $A \in \Real^{m \times n}$, there exists a pivot rule that limits the number of consecutive degenerate simplex pivots at any non-optimal extreme solution $x$ to at most $\min\{n-m-1, m-1\}$. }
    
    \textcolor{black}{The pivot rule can be efficiently implemented whenever an improving feasible direction of the form $\tilde x-x$, for some extreme point solution $\tilde x$, is available.}
\end{theorem}

\begin{proof}
Let $x$ be the currently considered basic feasible solution with basis $B$. If $y^{0}$  is chosen to be equal to $\tilde x- x$ for some basic feasible solution $\tilde x$ of our LP with $c^T \tilde x < c^T x$, then one can observe that $|Q_2(y^0, B)| \le m$ because $\tilde x$ has at most $m$ non-zero coordinates. By the proof of Theorem \ref{th:pivot-rule} above, the latter value strictly decreases with each degenerate step produced by the antistalling pivot rule. Hence, the total number of consecutive degenerate pivots at the vertex $x$ can be strengthened to be at most $\min\{n-m-1, m-1\}$. \qed
\end{proof}

\section{Exploiting the antistalling rule for general bounds}
\label{sec:applications}

Here we combine the result of the previous section with the analysis of \cite{kitahara2013,kitahara2014}. The authors of the latter works
give a bound on the number of non-degenerate simplex pivots that depends on $n,m$ and the maximum and the minimum non-zero coordinate of basic solutions. The combined analysis yields a similar bound on the total number of (both degenerate and non-degenerate) simplex pivots.

We need a few additional notations. For any vector $x$, we let $\supp(x)$ denote its support. We denote the smallest and the largest non-zero coordinate of any basic feasible solution of \eqref{eq:LP} by $\delta$ and $\Delta$, respectively. We let $x^\star$ be an optimal basic feasible solution of \eqref{eq:LP}. Finally, for a generic iteration $q$ of the simplex algorithm with basis $B^q$ and an improving feasible direction $y^q$, we let $\Delta^q_{\overline{c}} := \max_{\{i \in N^q \mid y^q_i > 0\}} -\overline{c}_{N^q, i}$. Note that $\Delta^q_{\overline{c}}$ is the absolute value of the reduced cost of the entering variable according to the antistalling pivot rule defined in the previous section when using $y^q$.

We make use of the following result from~\cite{kitahara2014}.
\begin{lemma}[Lemma 4 of \cite{kitahara2014}]\label{lem:framework}
    If there exists a constant $\lambda >0$ such that
    \begin{equation} \label{eq:framework}
        c^T (x^{q+1} - x^\star) \le \big( 1-\frac{1}{\lambda}\big)c^T (x^{q} - x^\star)
    \end{equation}
    holds for any consecutive distinct basic feasible solutions $x^q\neq x^{q+1}$ generated by the simplex algorithm (with any pivot rule), the total number of distinct basic feasible solutions encountered is at most
    \begin{equation*}
        (n-m)\Big\lceil\lambda \ln \frac{m\Delta}{\delta}\Big\rceil\,.
    \end{equation*}
\end{lemma}

Given (\ref{eq:LP}), apply the simplex algorithm with the antistalling pivot rule described in the previous section. In particular, at a general iteration $t$ of the algorithm, let $B^t,x^t$, and $y^t$ be the basis, the basic solution, and the improving feasible direction considered by the algorithm, respectively. Let $N^t := [n]\setminus B^t$.
If $x^{t} \neq x^{t-1}$ (i.e., we encounter $x^t$ for the first time, as a result of a non-degenerate pivot) let $y^t:=x^\star - x^t$. 
Observe that, once this is specified, the improving feasible direction is determined by the antistalling rule for all remaining degenerate pivots at $x^t$. 


 
\begin{lemma} \label{lem:useful} 
Let $x^t = x^{t+k}$ be the basic feasible solution associated with bases $B^t$ and $B^{t+k}$, and assume $x^{t+k} \neq x^{t+k+1}$.
The following holds:
 \begin{itemize}
\item[(a)] $c^Ty^{t+k} \le c^Ty^{t}$.
\item[(b)] $y^{t+k}_{i} = y^t_{i}$ for all $i \in N^{t+k} \cap\, \supp(y^{t+k})$.
\end{itemize}
\end{lemma}
 
 \begin{proof}
The statement in (a) follows from~(\ref{eq:next_sec}). The statement in (b) follows from the construction of our antistalling rule by induction: if Case III never occurs during the $k$ degenerate pivots, then $y^{t+k} = y^t$ and the statement holds trivially. Suppose that the situation of Case III appears for
$B^{t+j}, y^{t+j}$. The improving feasible direction  changes (as $y^{t+j+1} = y^{t+j} + \alpha z$ ), but among the non-basic coordinates this only affects the value of $y^{t+j+1}_f$ (where $f$ is the entering index). Immediately after, Case II occurs and $f$ gets pivoted in, while a basic index $i$ with $y^{t+j+1}_i=0$ gets pivoted out, i.e., $B^{t+j+1} = \big(B^{t+j}\setminus\{i\}\big)\cup\{f\}$ and $N^{t+j+1} = \big(N^{t+j}\setminus\{f\}\big)\cup\{i\}$. Hence, the statement holds.    \qed
 \end{proof}

The next result, which establishes \eqref{eq:framework} with $\lambda := \frac{(n-m)\Delta}{\delta}$ for the simplex algorithm with the antistalling pivot rule, is inspired by \cite{kitahara2014}[Theorem 3].

\begin{lemma} \label{lem:optgap} Let $x^t = x^{t+k}$ be the basic feasible solution associated with bases $B^t$ and $B^{t+k}$. Assume 
$x^{t-1} \neq x^{t}$ and $x^{t+k} \neq x^{t+k+1}$. We have
        \begin{equation*}
        c^Tx^{t+k+1}-c^Tx^\star \le \big( 1 - \frac{\delta}{(n-m)\Delta} \big) (c^Tx^{t+k}-c^Tx^\star)\,.
    \end{equation*}
\end{lemma}

\begin{proof}
    The optimality gap for $x^{t+k}$ can be bounded as follows:
    \begin{equation}\label{eq:optgap}
     \begin{array}{rcl}
        c^Tx^{t+k}-c^Tx^\star & = & c^Tx^{t}-c^Tx^\star \\
        & = & -c^Ty^t \\
        & \le &-c^Ty^{t+k} \\
        & = &-\overline{c}^T_{N^{t+k}}y^{t+k}_{N^{t+k}} \\
        & = & \sum_{i\in N^{t+k}} -\overline{c}^T_{N^{t+k},i} y^{t+k}_{i} \\
        & = & \sum_{i\in N^{t+k} \mid y^{t+k}_i>0} -\overline{c}^T_{N^{t+k},i} y^{t+k}_{i} \\
        & \le & (n-m)\Delta^{t+k}_{\overline{c}}  \Delta \,,
    \end{array}   
\end{equation}
where we used Lemma~\ref{lem:useful}(a) for the first inequality, and \eqref{eq:basis_decomp} for the third equality. The last equality follows from $y^{t+k}_i \ge 0, i \in N^{t+k}$ which is implied by feasibility of $y^{t+k}$ at $x^{t+k}$. The last inequality follows from $\max_{\{i \in N^{t+k} \mid y^{t+k}_i > 0\}} -\overline{c}_{N^{t+k}, i} =\Delta^{t+k}_{\overline{c}}$, the fact that $|N^{t+k}|=n-m$, and because $y^{t+k}_{i} = y^t_{i} = x^\star_i - x^t_i \le \Delta$  
holds for any $i \in N^{t+k} \cap\, \supp(y^{t+k})$, using Lemma~\ref{lem:useful}(b).

Let $x^{t+k}_{f}$ denote the entering variable at $(t+k)\textsuperscript{th}$ iteration. Note that $x^{t+k+1}_{f} \neq 0$ since $x^{t+k} \neq x^{t+k+1}$. Then,
\begin{equation*}
    \begin{array}{rl}
        c^Tx^{t+k} - c^Tx^{t+k+1} &= \overline{c}_{N^{t+k}}^T(x^{t+k} - x^{t+k+1})_{N^{t+k}} \\
        &= -\overline{c}_{N^{t+k},f} x^{t+k+1}_{f} \\
        &\ge \Delta^{t+k}_{\overline{c}} \delta\\
        &\ge \frac{\delta}{(n-m)\Delta}c^T(x^{t+k}-x^\star)
    \end{array}
\end{equation*}
hold, where we used \eqref{eq:basis_decomp} for the first equality and \eqref{eq:optgap} for the last inequality. Rearranging the terms yields the lemma statement.\qed
\end{proof}

Now, combining Lemma \ref{lem:framework} and Lemma \ref{lem:optgap} allows to conclude that the simplex algorithm with the antistalling pivot rule described in Section \ref{sec:antistalling} encounters at most $(n-m)\lceil\frac{(n-m)\Delta}{\delta}\ln\big(m\frac{\Delta}{\delta}\big)\rceil$ distinct basic feasible solutions. 
Since \textcolor{black}{Theorem \ref{th:pivot-rule2}} bounds the number of consecutive non-degenerate steps by 
$\min\{n-m-1,m-1\}$ (again, due to the choice of the improving feasible direction), combining it with the latter result yields a bound on the total number of pivots required to reach an optimal vertex. However, this bound does not take into account the number of (degenerate) pivots that might have to be performed at an optimal vertex to encounter a basis satisfying the optimality criterion. This is due to the fact that the antistalling pivot rule requires an improving feasible direction, which we do not have at an optimal vertex. Hence, we have to handle this case separately.

\begin{lemma}\textcolor{black}{
Let $B$ and $B^\star$ be a non-optimal and an optimal basis, respectively, both associated with an optimal basic feasible solution $x^\star$ of \eqref{eq:LP}. We can obtain $B^\star$ from $B$ by performing at most $n-m$ degenerate simplex pivots.}    
\end{lemma}

\begin{proof}
Since $B$ is not optimal, there exists $f\in N$ with $\overline{c}_{N, f} = c^Tz<0$ with $z$ as in (\ref{eq:ray}). Observe that there exists $i \in B \cap N^\star$ with $\overline{A}_{if}>0$ and $x_i=0$: otherwise, all coordinates of $z_{N^\star}$ would be non-negative and hence $z \in C_{N^\star}:=\{x\in \Real^n\mid Ax=0, x_{N^\star} \ge 0\}$. The latter however contradicts the fact that all extreme rays of the above cone $C_{N^\star}$ have non-negative scalar products $\overline{c}_{N^\star}$ with $c$ due to optimality of $B^\star$. Hence one could perform a simplex pivot on $B$ with entering variable $f$ and leaving variable $i$. Let $B^\prime := \big(B \setminus\{i\}\big) \cup \{f\}$ and $N^\prime := \big(N \setminus\{f\}\big) \cup \{i\}$. Note that $\overline{c}_{N^\prime,i} \ge 0$ and $i\in N^\prime \cap N^\star$. Either $B^\prime$ is a optimal basis and we stop, or there exists $j \in N^\prime \setminus \{i\}$ with $\overline{c}_{N^\prime, j}<0$. In the latter case, however, we can enforce the constraint $x_i = 0$ by removing the variable $x_i$ together with the corresponding column of $A$ and entry $c_i$ of $c$ from \eqref{eq:LP}. By doing so we obtain a new LP with the number of variables smaller by one that has $B^\prime$ and $B^\star$ as a non-optimal and an optimal basis, respectively, since  $\overline{c}_{N^\prime \setminus \{i\}, j}<0$ and $\overline{c}_{N^\star \setminus \{i\}} \ge 0$. We can set $B=B^\prime$ and repeat the process. Since at each iteration a variable with index from $N^\star$ gets removed, after at most $n-m$ iterations $B = B^\star$. \qed
\end{proof}

Since the number of degenerate pivots at an optimal vertex is bounded by $n-m$ \textcolor{black}{by the above lemma}, we can state the following result. 


\begin{theorem}\label{th:num-pivots}
Given an LP \eqref{eq:LP} and an initial feasible basis, there exists a pivot rule that makes the simplex algorithm reach an optimal basis in at most $$\min\{n-m,m\}(n-m)\lceil \frac{(n-m)\Delta}{\delta}\ln\big(m\frac{\Delta}{\delta}\big)\rceil$$ simplex pivots.
\end{theorem}


Observe that for integral $A$ and $b$ (which can be assumed without loss of generality for rational LPs), one can bound $\Delta \le ||b||_1\Delta_{A}$ and $\delta \ge \frac{1}{\Delta_{A}}$ due to Cramer's rule, where $\Delta_{A}$ is the largest absolute value of the determinant of a square submatrix of $A$. Then the next statement is a straightforward corollary of the above theorem. 

\begin{corollary} \label{cor:num-pivots}
    For any basic feasible solution of an LP \eqref{eq:LP} with integral $A$ and $b$, there exists a sequence of at most $$\min\{n-m,m\}(n-m)\lceil (n-m)\Delta_{A}^2||b||_1\ln\big(m\Delta_{A}^2||b||_1\big)\rceil$$ simplex pivots leading to an optimal basis.
\end{corollary}



\subsection{Application to combinatorial LPs}
We conclude this section by observing that, using the latter corollary, one can prove the existence of short sequences of simplex pivots 
(that is, of length strongly-polynomial in $n,m$) between any two extreme points of several combinatorial LPs, that is, LPs modeling the set of feasible solutions of famous combinatorial optimization problems. We report a few examples below.

\paragraph{(a) LPs modeling matching/vertex-cover/edge-cover/stable-set problems in bipartite graphs.} 
For matchings, the LP maximizes a linear function over a set of constraint of the form $\{A'x \leq 1, x \geq 0\}$, where the coefficient matrix $A'$ is the node-edge incidence matrix of an undirected bipartite graph. After putting the LP in standard equality form by adding slack variables, we get constraints of the form $\{Ax = 1, x \geq 0\}$, where $A$ is a totally unimodular matrix (and so $\Delta_A =1$). The result then follows from Corollary~\ref{cor:num-pivots}. The same holds for vertex-cover 
(minimizing a linear function over constraints of the form $\{{A'}^Tx \geq 1, x \geq 0\}$), edge-cover (minimizing a linear function over constraints of the form $\{{A'}x \geq 1,  x \geq 0\}$), stable-set (maximizing a linear function over constraints of the form $\{{A'}^Tx \leq 1,  x \geq 0\}$). 

\paragraph{(b) LPs modeling optimization over the \emph{fractional} matching/vertex-cover/edge-cover/stable-set polytopes.}
These correspond to the natural LP relaxations of the problems discussed in (a), for general graphs. The same LPs as in (a), in non-bipartite graphs, have a constraint matrix $A'$ (resp. ${A'}^T$) that is not totally unimodular. However, the set of constraints defines half-integral polytopes (see~\cite{appa2006,Nemhauser74}). Note that, after putting the LPs in standard equality form, the slack variables can be loosely bounded from above by the number of variables $n$. Hence,  $\Delta \le n$ and $\delta \ge \frac{1}{2}$, and the result follows from Theorem~\ref{th:num-pivots}.

\paragraph{(c) LP for the stable marriage problem.}
The classical stable marriage problem is defined on a bipartite graph where each node has a strict preference order over the set of its neighbours. One looks for a matching that does not contain any blocking pair, that is, a pair of nodes that mutually prefer each other with respect to their matched neighbour.   
There is an (exact) LP formulation for the problem~\cite{Rothblum,LP-bliss}, that has constraints of the form $\{A'x \leq 1, B'x \geq 1, x \geq 0\}$. 
Here $A'$ is again the node-edge incidence matrix of an undirected bipartite graph, while $B'$ is a matrix that stems from imposing an additional constraint for each edge $\{uv\}$, that essentially prevents $\{uv\}$ from being a blocking pair: 
$$x_{uv} + \sum_{w: w >_u v}   x_{uw} + \sum_{w: w >_v u  } x_{vw} \geq 1$$

In the above expression, $w >_u v$ (resp. $w >_v u$) means that $u$ prefers $w$ over $v$ (resp. $v$ prefers $w$ over $u$).
After putting the LP in standard equality form, the slack variables can be bounded by $1$. Hence,  $\Delta = \delta = 1$, and the result follows from Theorem~\ref{th:num-pivots}.

\smallskip
\noindent
\paragraph{(d) LPs modeling various flow problems (such as max flows, min cost flows, flow circulations) with unit (or bounded) capacity values.} 
Flow problems in capacitated graphs are modeled using LPs with constraints of the form
$\{{A'}x = b, l \leq x \leq u\}$
where the constraint matrix $A'$ is here a node-arc incidence matrix of a directed graph, which is totally unimodular (see~\cite{schrijver-book}). The right-hand-side vectors ($b, l,u$) can be bounded in terms of the total capacity values. Therefore, assuming these are bounded integers, one can rely on Corollary~\ref{cor:num-pivots} to get the result, similarly to (a).

\medskip
One can compute the corresponding pivot sequence for the problems mentioned in (a)-(d)  by running the simplex with the antistalling pivot rule. However, one needs to first solve the LP and obtain an optimal basic solution in order to use it to guide the antistalling pivot rule. 

We remark that the existence of a strongly polynomial sequence of pivots for general min-cost flows (hence for problems in (d)), was known already, as it follows from the fundamental work of Orlin \cite{Orlin97}, that applies even in the case of large capacity values. 
In comparison, our bound is weaker as it requires the capacity values to be bounded. 

Instead, despite being known to be solvable in strongly polynomial time, to the best of our knowledge the problems in (a)-(c) were not previously known to admit strongly polynomial bounds on the number of simplex pivots for their natural LP formulations.

\section{Computational experiments}\label{sec:comp}
In this section we report some computational experiments that we conducted to evaluate the performance of the antistalling pivot rule when applied to actual linear programs. \textcolor{black}{The code is now available at \url{https://github.com/kirill-kukharenko/antistalling}.}

For the implementation, we used the python package CyLP \cite{cylp} which wraps COIN-OR’s CLP solver \cite{coin-or} and provides tools for implementing a preferred pivot rule in python to substitute CLP’s built-in ones. \textcolor{black}{We used two different testsets} of diverse dimensions and sparsity. \textcolor{black}{The first one is} the benchmark Netlib LP dataset \cite{netlib}, \textcolor{black}{containing 108 LPs.} \textcolor{black}{The second one contains the root LPs of the benchmark set of MIPLIB 2017 \cite{miplib2027}, and is composed of 240 instances}. All our experiments were conducted on a laptop with Intel Core i7-13700H 2.40GHz CPU and 32GB RAM.

For each of the aforementioned LPs, we ran the simplex algorithm with our antistalling pivot rule, and in addition with other \textcolor{black}{6} well-known pivot rules (which \cite{cylp} provides implementations for). These are:
\begin{itemize}
\item Dantzig's rule: the entering variable is the one with the most negative reduced cost;
\item Steepest edge: the entering variable is the one which yields a direction $z$ maximizing $\frac{-c^Tz}{||z||_2}$ (i.e., maximizing the improvement normalized according to the 2-norm);
\item LIFO: the entering variable is the one that minimizes the number of iterations that have past since the variable left the basis;
\item Most Frequent: the entering variable is the one that maximizes the number of times it was previously chosen as the entering variable;
\item Bland's rule: the entering variable is the one with the lowest index.
\textcolor{black}{\item Positive edge rule: the entering variable is selected with a heuristic that tries to identify non-degenerate pivots with high probability~\cite{raymond2010positive}.}
\end{itemize}

We tracked the number of pivots required by each of the above rules \textcolor{black}{to solve the LPs to optimality (we did not perturb the instances)}. For our antistalling pivot rule, we guided it using the feasible direction $y=x^\star - x$ at any non-optimal vertex $x$ of the LP, for a pre-computed optimal basic solution $x^*$ (in particular, the one found by Dantzig's pivot rule).
Intuitively, this represents the best direction that can possibly guide our rule, as it is the direction leading immediately to an optimal solution. 
Our objective with the computational experiments is to see whether this choice translates into few pivots in practice, since $y$ might change during the degenerate steps and therefore in reality we do not have control on the actual edge-direction we end up moving along.

\textcolor{black}{We remark that we are here interested in the \emph{number of pivots} required by various pivot rules to solve LPs to optimality, and not in the \emph{computational time} needed to do so. In fact, the CyLP package connects to the solver through a relatively slow Python interface, in contrast to implementing the rules directly into a solver. Therefore, the efficiency of the pivot rule implementations, and hence the number of problems solved by each rule within the time/memory limit, can be heavily affected by this overhead. For a fair comparison, in this section we then only included the instances solved to optimality by all respective pivot rules.}
\textcolor{black}{However, for completeness, detailed tables with all the results of the computational experiments, for both Netlib and MIPLIB testsets, can be found in Appendices \ref{app:netlib} and \ref{app:miplib}, respectively.}

\textcolor{black}{Our first report deals with 84 out of 108 Netlib problems, for which each of the aforementioned 7 pivot rules was able to find an optimal solution within the time limit of 30 minutes and the memory limit of 32GB RAM. For the results, see Figure \ref{fig:netlib}.} 
\textcolor{black}{
Similarly, for MIPLIB this restriction leaves us with 117 out of 240 problems. See Figure \ref{fig:miplib} for the results.}

\begin{figure}
    \centering
    \includegraphics[width=0.95\textwidth]{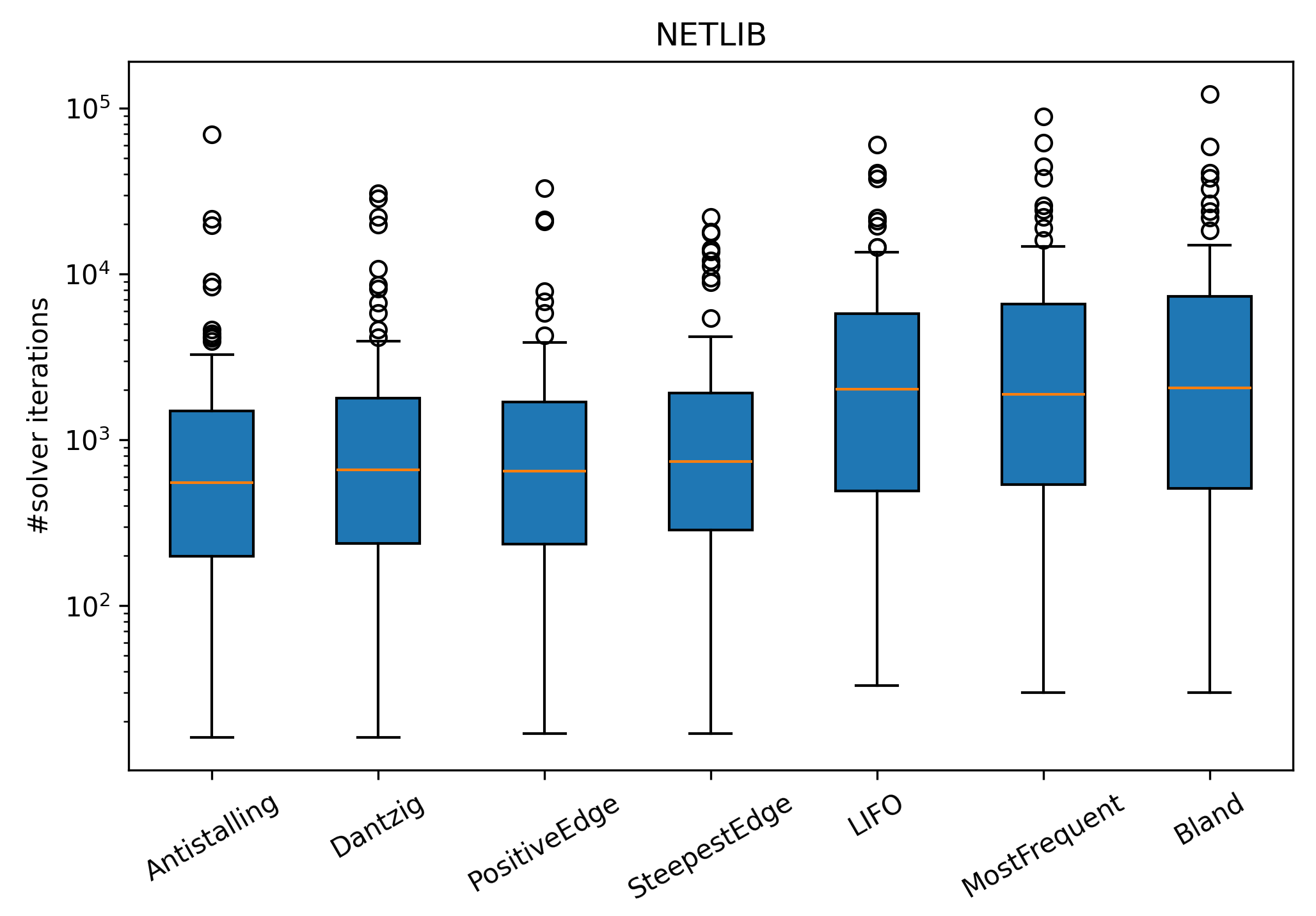}
    \caption{\textcolor{black}{Netlib dataset results. 
Data are presented using standard \emph{box plots,} which  display the distribution of a dataset based on its quartiles. Specifically, the box extends from the first quartile (Q1) to the third quartile (Q3), representing the middle $50\%$ of the data. 
The thin vertical lines (whiskers) extend from the box to the most extreme data points lying within 
1.5 times the interquartile range (Q3-Q1). Points beyond the whiskers are shown individually as outliers (fliers).
Orange horizontal lines denote the medians: 551,  659,   648,  739, 2027,  1890, and 2049, respectively.}}
    \label{fig:netlib}
\end{figure}

\begin{figure}
    \centering
    \includegraphics[width=0.95\textwidth]{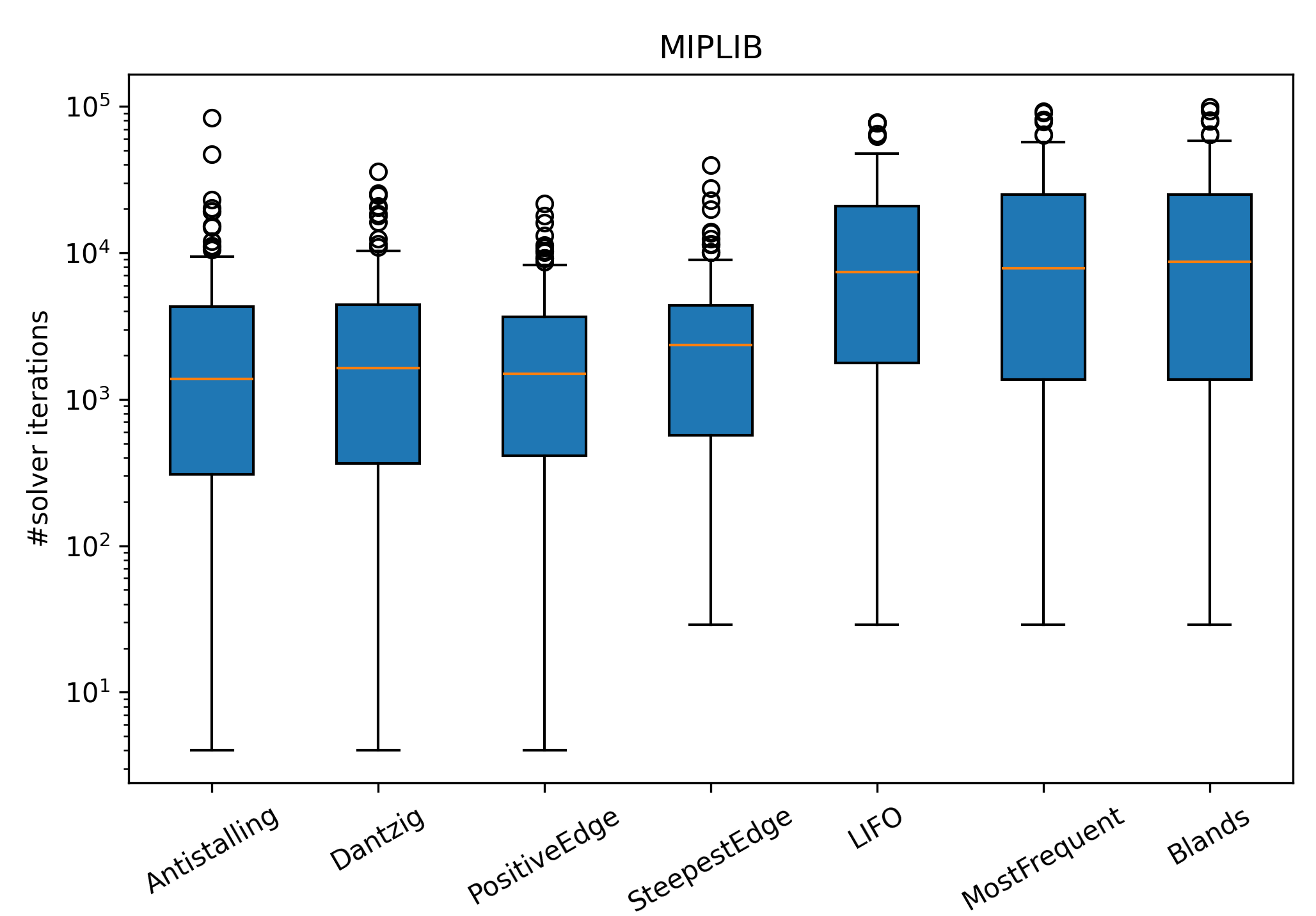}
    \caption{\textcolor{black}{MIPLIB dataset results. Orange lines denote the medians:} 
    \textcolor{black}{1376, 1633, 1493, 2335, 7374, 7829, and 8641,}
    \textcolor{black}{respectively.}}
    \label{fig:miplib}
\end{figure}

The computational experiments show that the antistalling pivot rule, guided by a known optimal solution, actually performs quite well, \textcolor{black}{especially on the Netlib instances.} Most often it manages to find a relatively short sequence of pivots to an optimal basis compared to the other pivot rules.  We highlight \textcolor{black}{for example} the problem\textcolor{black}{s} 25fv47 \textcolor{black}{from Netlib and istanbul-no-cutoff.mps from MIPLIB}, where \textcolor{black}{the antistalling pivot rule required less than one fourth of the number of iterations of any other rule.} On the other hand, there \textcolor{black}{are} instances\textcolor{black}{, e.g., fit2p from Netlib and ic97\_potential.mps from MIPLIB, where the antistalling pivot rule actually showed the worst performance compared to all its competitors. }

\textcolor{black}{
Overall, the results seem to confirm that guidance by an optimal solution indeed often translates into a small number of pivots, motivating further theoretical studies on how such a direction could be approximately followed.
}

\section{Final Remarks}\label{sec:open} 
We would like to conclude the paper with some remarks and potential research directions that originate from our findings.

The results of this work show the existence of a short sequence of degenerate simplex pivots that leads out of degeneracy.
It would be interesting to show that similar results can be achieved through perturbations. Informally, perturbing the right-hand-side of the constraints has the effect of ``splitting'' a degenerate vertex into a set of (possibly exponentially many) distinct vertices, each associated to a distinct basis. 
 It is legitimate to ask whether every new vertex allows for a short monotone  path (where monotonicity is with respect to the linear objective function of the LP), which leads to a vertex out of this set. Note that the existence of the short degeneracy-escaping sequence of simplex pivots showed in this work, generally does not answer the latter question since some bases used in the sequence might become infeasible after perturbation.

Next, we would like to connect the above question to the recent work \cite{kaibelkukh2023}. Given a basis $B$ defining a degenerate vertex $v$, one could imagine to separate $v$ from all its neighbouring vertices by a hyperplane: this way, one would obtain a pyramid (with $v$ at its apex). Suppose this pyramid is full-dimensional and is described by some $Ax \le b$: then, the above question is equivalent to perturbing the side facets of the pyramid to allow for a short path from the new apex defined by the given basis $B$, to any of the vertices in the bottom of the pyramid. This can be done relying on the techniques used for the construction of a so-called rock extension \cite{kaibelkukh2023} (when viewing the base of the pyramid as the $d-1$-dimensional original polytope, and the pyramid as its $d$-dimensional extension). Such rock extension has been proven to have the property that the apex is connected to any bottom vertex by a path of at most linear length. However these paths are not necessarily monotone for a given linear objective function, and moreover, the construction of the rock extension requires not only perturbing the right-hand-side $b$ but the left-hand-side $A$ as well.

Another future research direction, that seems to be supported also by the performed computational experiments, would be to analyze the relation between the  improving feasible direction $y$ that is required by our antistalling pivot rule, and the actual edge-direction $z$ along which one ends up moving after all the degenerate pivots. It would be very interesting to identify some conditions which ensure that $z$ is a good approximation of $y$, e.g., in terms of objective function's improvement.

Furthermore, as it was mentioned in Section \ref{sec:antistalling}, in order for our bound to hold, the only requirement we need to impose when performing a pivot is that the entering variable lies in the support of $y$ (besides of course having negative reduced cost).  In order to obtain the results from Section \ref{sec:applications}, where both degenerate and non-degenerate pivots have to be taken into account, we chose the entering variable to be the one with the most negative reduced cost, i.e., we used Dantzig's rule just restricted to the coordinates in $\supp(y)$. 
Of course, one could think of analyzing the performance of the antistalling pivot rule with other choices of entering coordinates from $\supp(y)$, e.g., according to steepest edge or shadow vertex. Variations of these latter two rules, in particular, have been shown to play a key role in the context of 0/1 polytopes~\cite{black2021}. 

\subsubsection*{Acknowledgements.} The authors are grateful to Volker Kaibel for helpful discussions on the geometric interpretation of the antistalling pivot rule. We thank Alex Black and anonymous reviewers for their helpful comments. The authors express gratitude to the Deutsche Forschungsgemeinschaft (DFG, German Research Foundation) for supporting the first author within 314838170, GRK 2297 MathCoRe and providing travel funds for visiting the second author in Milan. The second author is supported by the NWO VIDI grant VI.Vidi.193.087.

%
%
%
\bibliographystyle{splncs04}
\bibliography{refs_1}
\newpage
\appendix
\section{Detailed Netlib results}\label{app:netlib}

\begin{longtable}{llllllll}
\caption{Netlib LPs. * signifies the instance couldn't be solved by the respective pivot rule within the time and memory limit.} \label{tab:netlib} \\
\toprule
problem & Dantzig & PositiveEdge & LIFO & MostFrequent & SteepestEdge & Bland & Antistalling \\
\midrule
\endfirsthead
\caption[]{Netlib} \\
\toprule
problem & Dantzig & PositiveEdge & LIFO & MostFrequent & SteepestEdge & Bland & Antistalling \\
\midrule
\endhead
\midrule
\multicolumn{8}{r}{Continued on next page} \\
\midrule
\endfoot
\bottomrule
\endlastfoot
25fv47.mps & 8617 & 7892 & 9934 & 10016 & 9530 & 9870 & 1468 \\
80bau3b.mps & 19858 & 20774 & 37525 & 38130 & 17628 & 37965 & 4119 \\
adlittle.mps & 111 & 125 & 197 & 316 & 98 & 326 & 56 \\
afiro.mps & 16 & 17 & 33 & 30 & 17 & 30 & 16 \\
agg.mps & 130 & 134 & 300 & 228 & 228 & 233 & 128 \\
agg2.mps & 216 & 188 & 210 & 194 & 258 & 200 & 228 \\
agg3.mps & 239 & 166 & 317 & 236 & 276 & 249 & 136 \\
bandm.mps & 674 & 770 & 2522 & 2762 & 1283 & 2694 & 674 \\
beaconfd.mps & 104 & 102 & 117 & 117 & 195 & 117 & 89 \\
blend.mps & 92 & 86 & 223 & 299 & 122 & 253 & 82 \\
bnl1.mps & 3290 & 2228 & 6295 & 6408 & 1883 & 6412 & 684 \\
bnl2.mps & 12115 & * & 19314 & * & * & * & 1524 \\
boeing1.mps & 675 & 695 & 2485 & 2570 & 842 & 2726 & 401 \\
boeing2.mps & 174 & 163 & 518 & 550 & 300 & 568 & 230 \\
bore3d.mps & 198 & 199 & 449 & 496 & 291 & 424 & 152 \\
brandy.mps & 353 & 327 & 1595 & 1701 & 1060 & 1620 & 182 \\
capri.mps & 506 & 479 & 561 & 1148 & 416 & 983 & 334 \\
cre-a.mps & 5828 & 3202 & 19458 & 25810 & 14105 & 26623 & 1554 \\
cre-b.mps & 326739 & 31408 & 46413 & 295815 & * & 322064 & 10028 \\
cre-c.mps & 6712 & 3504 & 20850 & 24339 & 17947 & 23915 & 1483 \\
cre-d.mps & 316546 & 22426 & 73494 & 265573 & * & 281915 & 38690 \\
cycle.mps & 5383 & 5152 & * & * & * & * & 4621 \\
czprob.mps & 1905 & 2557 & 14585 & 14724 & 1373 & 14899 & 886 \\
d2q06c.mps & 29315 & 29346 & 30883 & * & * & 31571 & 5661 \\
d6cube.mps & 22028 & 21286 & 21790 & 22154 & 22061 & 21891 & 21569 \\
degen2.mps & 1891 & 1342 & 3480 & 3431 & 729 & 3564 & 627 \\
degen3.mps & 10713 & 5816 & 12145 & 12312 & 2839 & 11822 & 4241 \\
dfl001.mps & 104120 & 99785 & 104924 & 105564 & * & 97847 & 73090 \\
e226.mps & 718 & 561 & 1790 & 1820 & 767 & 1806 & 216 \\
etamacro.mps & 731 & 609 & 2652 & * & * & * & 995 \\
fffff800.mps & 868 & 671 & 4169 & 4922 & 900 & 4943 & 1216 \\
finnis.mps & 611 & 488 & 810 & 1385 & 523 & 1163 & 347 \\
fit1d.mps & 2636 & 2369 & 3301 & 3290 & 2532 & 3377 & 3254 \\
fit1p.mps & 2193 & 1888 & 5656 & 7214 & 1340 & 7246 & 4612 \\
fit2d.mps & 29994 & 31236 & 31841 & 31923 & * & 31880 & 9417 \\
fit2p.mps & 30716 & 33042 & 60330 & 61925 & 11275 & 58919 & 69711 \\
forplan.mps & 243 & 199 & * & 2065 & * & 2055 & 136 \\
ganges.mps & 1628 & 1834 & 2237 & 2521 & 1801 & 2217 & 1728 \\
gfrd-pnc.mps & 1238 & 1030 & 1501 & 1626 & 1062 & 1813 & 801 \\
greenbea.mps & 28646 & 20538 & * & * & * & 33255 & 26251 \\
greenbeb.mps & 30340 & 13041 & 33436 & 35418 & * & * & 30644 \\
grow15.mps & 725 & 647 & 2964 & 2217 & 2455 & 3016 & 1524 \\
grow22.mps & 1038 & 904 & 4612 & 3410 & 5427 & 5263 & 2909 \\
grow7.mps & 306 & 260 & 1157 & 702 & 838 & 941 & 469 \\
israel.mps & 331 & 285 & 453 & 747 & 431 & 887 & 123 \\
kb2.mps & 61 & 65 & 167 & 127 & 70 & 128 & 55 \\
ken-07.mps & 4622 & 3825 & 11384 & 18902 & 3748 & 18336 & 3929 \\
ken-11.mps & 31812 & 29710 & 113471 & 114524 & * & 114656 & 40987 \\
ken-13.mps & 222253 & 91988 & 228878 & 230157 & * & 229959 & * \\
ken-18.mps & * & * & * & * & * & * & * \\
lotfi.mps & 207 & 287 & 1573 & 1558 & 295 & 1653 & 121 \\
maros.mps & 3922 & 2559 & 8837 & 9335 & 3880 & 8738 & 4375 \\
maros-r7.mps & 4144 & 4280 & 40649 & 44360 & 4176 & 40737 & 4144 \\
modszk1.mps & 1125 & 1082 & 4478 & 7128 & 1327 & 7132 & 672 \\
nesm.mps & 4392 & 4392 & * & 11946 & * & * & 1508 \\
osa-07.mps & 600 & 1293 & 2402 & 2205 & 750 & 4504 & 1090 \\
osa-14.mps & 1400 & 2510 & 5553 & 4787 & 1500 & 8807 & 806 \\
osa-30.mps & 2617 & 4925 & 11199 & 14545 & * & 30277 & 1551 \\
osa-60.mps & 5234 & 8905 & 22634 & 149478 & * & 78809 & 2966 \\
pds-02.mps & 1776 & 1668 & 10942 & 16084 & 2099 & 32573 & 19604 \\
pds-06.mps & 28565 & 6813 & 39989 & 88880 & 8904 & 122043 & 8402 \\
pds-10.mps & 122401 & 13140 & 65131 & 154595 & * & 208463 & 18440 \\
perold.mps & 6339 & 5416 & 7558 & * & * & 7697 & 1792 \\
pilot.ja.mps & 5565 & * & 10530 & * & * & * & 1354 \\
pilot.mps & * & * & * & * & * & * & * \\
pilot.we.mps & 8145 & 3862 & 14524 & 14493 & 13676 & 14690 & 9037 \\
pilot4.mps & 1326 & 1636 & 4895 & * & * & * & 1409 \\
pilot87.mps & * & * & * & * & * & * & * \\
pilotnov.mps & 1812 & 2582 & 12539 & 12187 & 12088 & 12530 & 907 \\
recipe.mps & 48 & 48 & 64 & 64 & 48 & 64 & 48 \\
sc105.mps & 103 & 102 & 148 & 120 & 102 & 119 & 108 \\
sc205.mps & 231 & 232 & 294 & 252 & 234 & 228 & 224 \\
sc50a.mps & 49 & 45 & 56 & 53 & 49 & 53 & 47 \\
sc50b.mps & 49 & 50 & 51 & 50 & 49 & 50 & 50 \\
scagr25.mps & 653 & 834 & 3311 & 3042 & 695 & 3081 & 771 \\
scagr7.mps & 140 & 186 & 501 & 361 & 190 & 401 & 154 \\
scfxm1.mps & 457 & 530 & 1826 & 1960 & 613 & 1694 & 537 \\
scfxm2.mps & 993 & 1053 & 4854 & 4990 & 1380 & 5010 & 2380 \\
scfxm3.mps & 1534 & 1575 & 7462 & 7499 & 2015 & 7519 & 1928 \\
scorpion.mps & 364 & 366 & 782 & 444 & 363 & 450 & 382 \\
scrs8.mps & 667 & 809 & 5386 & 5334 & 555 & 5474 & 319 \\
scsd1.mps & 272 & 237 & 2228 & 2753 & 180 & 2697 & 202 \\
scsd6.mps & 620 & 517 & 4913 & 4845 & 602 & 4964 & 237 \\
scsd8.mps & 1918 & 1307 & 10620 & 10652 & 1719 & 10563 & 2037 \\
sctap1.mps & 317 & 356 & 718 & 822 & 368 & 1156 & 191 \\
sctap2.mps & 1225 & 1270 & 1095 & 1352 & 1533 & 2077 & 658 \\
sctap3.mps & 1907 & 1756 & 1482 & 1626 & 2231 & 1586 & 944 \\
seba.mps & 560 & 649 & 508 & 510 & 876 & 507 & 976 \\
share1b.mps & 247 & 290 & 1156 & 1148 & 211 & 947 & 302 \\
share2b.mps & 157 & 131 & 245 & 368 & 156 & 408 & 216 \\
shell.mps & 733 & 744 & 1216 & 1268 & 680 & 1406 & 779 \\
ship04l.mps & 323 & 343 & 1649 & 1637 & 368 & 2022 & 366 \\
ship04s.mps & 313 & 328 & 1028 & 1018 & 346 & 1159 & 363 \\
ship08l.mps & 598 & 628 & 2469 & 2534 & 588 & 3413 & 540 \\
ship08s.mps & 566 & 552 & 1328 & 1335 & 551 & 1628 & 496 \\
ship12l.mps & 1111 & 1017 & 6077 & 6234 & 988 & 8259 & 1152 \\
ship12s.mps & 951 & 914 & 3005 & 3558 & 919 & 2959 & 881 \\
sierra.mps & 1109 & 1134 & 970 & 1029 & 1864 & 1020 & 563 \\
stair.mps & 489 & 507 & 2884 & 800 & 503 & 728 & 1842 \\
standata.mps & 69 & 67 & 118 & 123 & 70 & 127 & 63 \\
standgub.mps & 69 & 67 & 118 & 123 & 70 & 127 & 80 \\
standmps.mps & 333 & 379 & 470 & 829 & 342 & 999 & 208 \\
stocfor1.mps & 79 & 79 & 241 & 268 & 79 & 247 & 79 \\
stocfor2.mps & 2001 & 2212 & 13526 & 13608 & 3591 & 13492 & 1774 \\
tuff.mps & 1655 & 720 & 2965 & 2991 & 701 & 2945 & 1011 \\
vtp.base.mps & 230 & 140 & 381 & 572 & 175 & 513 & 160 \\
wood1p.mps & 665 & 404 & 9584 & 9337 & 3282 & 9208 & 381 \\
woodw.mps & 2401 & 1877 & * & 32532 & * & 31839 & 2491 \\
\end{longtable}

\newpage

\newgeometry{textwidth=17cm}
\section{Detailed MIPLIB results} \label{app:miplib}
\begin{longtable}{llllllll}
\caption{MIPLIB root LPs. * signifies that the instance couldn't be solved by the respective pivot rule within the time and memory limit.} \label{tab:miplib} \\
\toprule
problem & Dantzig & PositiveEdge & LIFO & MostFrequent & SteepestEdge & Blands & Antistalling \\
\midrule
\endfirsthead
\caption[]{MIPLIB}  \\
\toprule
problem & Dantzig & PositiveEdge & LIFO & MostFrequent & SteepestEdge & Blands & Antistalling \\
\midrule
\endhead
\midrule
\multicolumn{8}{r}{Continued on next page} \\
\midrule
\endfoot
\bottomrule
\endlastfoot
30n20b8.mps & 354 & 413 & 29832 & 51208 & 8610 & 58208 & 294 \\
50v-10.mps & 282 & 272 & 1114 & 1037 & 2729 & 1027 & 332 \\
CMS750\_4.mps & 16295 & 13901 & 96361 & 100745 & * & 101308 & 41850 \\
academictimetablesmall.mps & 40401 & 26544 & 171784 & 176255 & * & * & 8027 \\
air05.mps & 18529 & 17937 & 24638 & 25176 & 2531 & 25058 & 19200 \\
app1-1.mps & 2602 & 2488 & 11247 & 22824 & 7809 & 22766 & 3886 \\
app1-2.mps & 26462 & 25375 & * & * & * & * & * \\
assign1-5-8.mps & 314 & 334 & 1235 & 1199 & 496 & 1149 & 239 \\
atlanta-ip.mps & * & 144325 & * & * & * & * & * \\
b1c1s1.mps & 8515 & 7124 & 24663 & 27554 & 13659 & 27456 & 2649 \\
bab2.mps & * & * & * & * & * & * & * \\
bab6.mps & 88009 & 57985 & * & * & * & * & * \\
beasleyC3.mps & 1313 & 1091 & 5691 & 3444 & 2403 & 2813 & 1430 \\
binkar10\_1.mps & 1519 & 1493 & 6279 & 9385 & 1714 & 10168 & 1376 \\
blp-ar98.mps & 750 & 806 & 10794 & 14173 & 2897 & 52747 & 244 \\
blp-ic98.mps & 383 & 353 & 10432 & 15684 & 2125 & 43764 & 90  \\
bnatt400.mps & 2098 & 1085 & 5055 & 24275 & 1128 & 18792 & 2144 \\
bnatt500.mps & 3257 & 1616 & 6261 & 34694 & 1326 & 33572 & 6050 \\
bppc4-08.mps & 267 & 307 & 5373 & 4064 & 3543 & 3194 & 203 \\
brazil3.mps & 224690 & 143249 & 228236 & * & * & * & * \\
buildingenergy.mps & * & * & * & * & * & * & * \\
cbs-cta.mps & 3915 & 4629 & 44014 & 19417 & * & 19132 & 3018 \\
chromaticindex1024-7.mps & 112861 & * & * & * & * & * & * \\
chromaticindex512-7.mps & 53453 & * & * & * & * & * & 24224 \\
cmflsp50-24-8-8.mps & 62624 & 29725 & 66646 & 66879 & * & 66483 & 46395 \\
co-100.mps & 1810 & 1432 & 27194 & 155300 & * & 156246 & 79949 \\
cod105.mps & 5043 & 10413 & 10806 & * & 2350 & * & 3517 \\
comp07-2idx.mps & 36630 & 19996 & 120086 & 140436 & * & 140131 & 124969 \\
comp21-2idx.mps & 17908 & 11221 & 77106 & 90612 & 3935 & 99479 & 83129 \\
cost266-UUE.mps & 1612 & 1782 & 3310 & 3582 & 4375 & 3626 & 1827 \\
cryptanalysiskb128n5obj14.mps & 83012 & 94390 & * & 125246 & * & * & * \\
cryptanalysiskb128n5obj16.mps & 83670 & 91007 & * & 142608 & * & * & * \\
csched007.mps & 1440 & 1902 & 7905 & 8245 & 1777 & 7794 & 603 \\
csched008.mps & 1012 & 1024 & 6211 & 6404 & 1265 & 6246 & 1647 \\
cvs16r128-89.mps & 28014 & 28833 & 26663 & 25872 & * & 28176 & 27912 \\
dano3\_3.mps & 50506 & * & * & * & * & * & * \\
dano3\_5.mps & 50506 & * & * & * & * & * & * \\
decomp2.mps & 2644 & 2574 & 7315 & 7555 & 3674 & 12070 & 5481 \\
drayage-100-23.mps & 4423 & 2050 & 15077 & 27251 & 2049 & 38281 & 6012 \\
drayage-25-23.mps & 4599 & 2189 & 14278 & 30007 & 1812 & 24166 & 6277 \\
dws008-01.mps & 313 & 246 & 1261 & 527 & 430 & 478 & 358 \\
eil33-2.mps & 233 & 241 & 1600 & 2129 & 151 & 2469 & 603 \\
eilA101-2.mps & 3325 & 3514 & 36453 & 79566 & 856 & * & 56899 \\
enlight\_hard.mps & 29 & 29 & 29 & 29 & 29 & 29 & 29 \\
ex10.mps & * & 179784 & * & * & * & * & * \\
ex9.mps & * & 84528 & * & * & * & * & * \\
exp-1-500-5-5.mps & 602 & 1031 & 639 & 639 & 1311 & 639 & 457 \\
fast0507.mps & 27383 & 24272 & 196331 & 197020 & * & 197160 & 22921 \\
fastxgemm-n2r6s0t2.mps & 5691 & 3666 & 17063 & 22611 & 2964 & 19282 & 5782 \\
fhnw-binpack4-4.mps & 665 & 667 & 960 & 625 & 1230 & 615 & 994 \\
fhnw-binpack4-48.mps & 9521 & 7946 & 8830 & 4430 & 19723 & 4406 & 7236 \\
fiball.mps & 6688 & 8910 & 12339 & 79670 & * & 73548 & 4025 \\
gen-ip002.mps & 71 & 71 & 147 & 90 & 86 & 91 & 32 \\
gen-ip054.mps & 61 & 61 & 152 & 104 & 49 & 105 & 30 \\
germanrr.mps & 3879 & 4119 & * & 56995 & * & * & 9180 \\
gfd-schedulen180f7d50m30k18.mps & * & * & * & * & * & * & * \\
glass-sc.mps & 1589 & 2876 & 20767 & 21471 & 499 & 22452 & 610 \\
glass4.mps & 93 & 86 & 304 & 235 & 116 & 187 & 73 \\
gmu-35-40.mps & 1651 & 765 & 2876 & 1947 & 2183 & 2582 & 1494 \\
gmu-35-50.mps & 1955 & 1041 & 3785 & 3745 & 2992 & 4648 & 2405 \\
graph20-20-1rand.mps & 3843 & 2835 & 24081 & * & 2517 & 12187 & 4597 \\
graphdraw-domain.mps & 268 & 309 & 622 & 304 & 386 & 292 & 168 \\
h80x6320d.mps & 364 & 263 & 10938 & 8128 & 6955 & 8835 & 499 \\
highschool1-aigio.mps & * & * & * & * & * & * & * \\
hypothyroid-k1.mps & 10923 & 11179 & 26362 & 15838 & 7999 & 24741 & 4307 \\
ic97\_potential.mps & 634 & 598 & 752 & 514 & 636 & 519 & 1101 \\
icir97\_tension.mps & 3186 & 2145 & 3432 & 4216 & 3587 & 3547 & 3383 \\
irish-electricity.mps & * & * & * & * & * & * & * \\
irp.mps & 225 & 287 & * & 4075 & 139 & 23929 & 191 \\
istanbul-no-cutoff.mps & 11975 & 11678 & 82444 & 86299 & * & 80471 & 2781 \\
k1mushroom.mps & 81426 & 16692 & 18513 & 5144 & * & 5400 & * \\
lectsched-5-obj.mps & 154064 & * & 19026 & 65039 & * & 64952 & 45012 \\
leo1.mps & 1906 & 1042 & 10858 & 22634 & 10125 & 22835 & 275 \\
leo2.mps & 4766 & 2269 & 17501 & 36383 & 12522 & 36488 & 193 \\
lotsize.mps & 1918 & 1785 & 3498 & 3567 & 3148 & 3599 & 2010 \\
mad.mps & 37 & 87 & 97 & 154 & 52 & 226 & 40 \\
map10.mps & 91211 & * & * & * & * & * & * \\
map16715-04.mps & * & * & * & * & * & * & * \\
markshare2.mps & 39 & 39 & 97 & 110 & 62 & 107 & 39 \\
markshare\_4\_0.mps & 20 & 20 & 39 & 42 & 31 & 37 & 19 \\
mas74.mps & 129 & 129 & 516 & 566 & 307 & 563 & 23 \\
mas76.mps & 79 & 79 & 414 & 500 & 110 & 528 & 37 \\
mc11.mps & 1677 & 1894 & 3430 & 2643 & 4578 & 2592 & 1327 \\
mcsched.mps & 4708 & 5116 & 11932 & 11842 & 10099 & 11784 & 5820 \\
mik-250-20-75-4.mps & 75 & 75 & 75 & 75 & 75 & 75 & 75 \\
milo-v12-6-r2-40-1.mps & 3739 & 5212 & 26261 & 15404 & 4622 & 9373 & 2451 \\
momentum1.mps & 12355 & 8999 & * & * & * & * & 11537 \\
mushroom-best.mps & * & 57526 & 10307 & 46970 & 1123 & * & * \\
mzzv11.mps & 39840 & 13789 & 73458 & 88323 & * & 106101 & 3987 \\
mzzv42z.mps & 12448 & 9094 & 62218 & 56988 & 11379 & 80630 & 3614 \\
n2seq36q.mps & 4060 & 2892 & 77923 & 81065 & 3048 & 79541 & 4388 \\
n3div36.mps & 143 & 150 & 2276 & 5299 & 295 & 8577 & 37 \\
n5-3.mps & 1396 & 995 & 11634 & 11481 & 2012 & 11550 & 890 \\
neos-631710.mps & 31810 & 17901 & * & * & * & * & * \\
neos-662469.mps & 30106 & 23253 & 61914 & 63061 & * & 62946 & 50345 \\
neos-787933.mps & 2163 & 2246 & * & 9279 & * & 8700 & 2164 \\
neos-827175.mps & 171542 & * & * & 165678 & * & 165157 & * \\
neos-848589.mps & 2135 & * & * & 13758 & * & * & 2114 \\
neos-860300.mps & 1155 & 893 & 7698 & 7829 & 1321 & 7938 & 224 \\
neos-873061.mps & 110262 & * & * & * & * & * & 33 \\
neos-911970.mps & 129 & 133 & 3183 & 3363 & 214 & 3447 & 92 \\
neos-933966.mps & 158717 & 167025 & 169776 & 155596 & * & 158554 & 109968 \\
neos-950242.mps & 6928 & 7299 & 64874 & 80748 & 2590 & 5921 & 2630 \\
neos-957323.mps & 28302 & 12481 & * & * & * & * & * \\
neos-960392.mps & 9058 & 10795 & 203203 & 205651 & * & 206419 & 8805 \\
neos17.mps & 936 & 767 & 3430 & 1605 & 570 & 1379 & 700 \\
neos5.mps & 105 & 107 & 193 & 122 & 99 & 270 & 7 \\
neos8.mps & 712 & 429 & 4581 & 6209 & 553 & 9684 & 1124 \\
neos859080.mps & 4 & 4 & 93 & 91 & 83 & 90 & 4 \\
neos-1122047.mps & 8685 & 8825 & * & * & * & * & 9201 \\
neos-1171448.mps & 79780 & 71487 & 59813 & 6001 & * & 5843 & 18055 \\
neos-1171737.mps & 25462 & 21642 & 21740 & 2542 & 22777 & 2630 & 22927 \\
neos-1354092.mps & 90484 & 88911 & 67471 & 67230 & * & 61909 & 15903 \\
neos-1445765.mps & 6122 & 6555 & 28656 & 31982 & 10084 & 46889 & 5772 \\
neos-1456979.mps & 35773 & 8662 & 30198 & 35218 & 2142 & 35123 & 3378 \\
neos-1582420.mps & 3670 & 2964 & 47584 & 63690 & 6191 & 64109 & 4293 \\
neos-2075418-temuka.mps & * & * & * & * & * & * & * \\
neos-2657525-crna.mps & 360 & 200 & 627 & 699 & 305 & 678 & 292 \\
neos-2746589-doon.mps & 120878 & 36276 & * & 271882 & * & * & 14163 \\
neos-2978193-inde.mps & 928 & 2709 & 37226 & 64529 & 2335 & 64472 & 833 \\
neos-2987310-joes.mps & 14222 & 13281 & 23023 & 25572 & * & 26282 & 15458 \\
neos-3004026-krka.mps & 6660 & 6660 & 6305 & 6305 & * & 6305 & 11610 \\
neos-3024952-loue.mps & 3645 & 3863 & 13076 & 4545 & 8948 & 5767 & 3073 \\
neos-3046615-murg.mps & 169 & 169 & 235 & 235 & 455 & 235 & 169 \\
neos-3083819-nubu.mps & 2678 & 2409 & 8224 & 11544 & 2240 & 15352 & 1681 \\
neos-3216931-puriri.mps & * & 34396 & 45197 & 45797 & 9782 & 50236 & * \\
neos-3381206-awhea.mps & 1584 & 1774 & 2196 & 2312 & 1978 & 2227 & 2002 \\
neos-3402294-bobin.mps & 1066 & 1023 & 3374 & * & 461 & * & 1134 \\
neos-3402454-bohle.mps & 5481 & 3457 & * & * & * & * & * \\
neos-3555904-turama.mps & 6076 & 1607 & * & * & * & * & 4329 \\
neos-3627168-kasai.mps & 1633 & 1900 & 4127 & 3905 & 5382 & 5776 & 1217 \\
neos-3656078-kumeu.mps & 20525 & 17539 & 113007 & 110161 & * & 109498 & 9260 \\
neos-3754480-nidda.mps & 291 & 430 & 2202 & 974 & 635 & 1035 & 231 \\
neos-3988577-wolgan.mps & 40469 & 24658 & * & * & * & * & * \\
neos-4300652-rahue.mps & 85331 & 17034 & 48186 & * & * & * & 26589 \\
neos-4338804-snowy.mps & 683 & 692 & 441 & 441 & 792 & 441 & 1681 \\
neos-4387871-tavua.mps & 2292 & 2471 & 9202 & 6459 & 3000 & 26601 & 3427 \\
neos-4413714-turia.mps & 27839 & 10821 & * & * & * & * & 16425 \\
neos-4532248-waihi.mps & 3191 & 982 & * & * & * & * & 7981 \\
neos-4647030-tutaki.mps & 11605 & 12066 & 18514 & 18452 & * & 18167 & * \\
neos-4722843-widden.mps & 14983 & 10825 & * & 30597 & * & 37043 & * \\
neos-4738912-atrato.mps & 4027 & 5190 & 26204 & 26796 & 27602 & 26215 & 5533 \\
neos-4763324-toguru.mps & 302123 & 71563 & * & * & * & * & * \\
neos-4954672-berkel.mps & 962 & 965 & 3574 & 3792 & 3055 & 4385 & 670 \\
neos-5049753-cuanza.mps & * & * & * & * & * & * & * \\
neos-5052403-cygnet.mps & * & * & * & * & * & * & * \\
neos-5093327-huahum.mps & 19617 & 20575 & * & * & * & * & * \\
neos-5104907-jarama.mps & * & * & * & * & * & * & * \\
neos-5107597-kakapo.mps & 6191 & 6969 & 14446 & 11175 & * & 10175 & 8147 \\
neos-5114902-kasavu.mps & * & * & * & * & * & * & * \\
neos-5188808-nattai.mps & 26125 & 12571 & 24862 & 46541 & * & * & 9429 \\
neos-5195221-niemur.mps & 43996 & 14654 & 21534 & * & * & * & 13985 \\
net12.mps & 3929 & 5043 & 29848 & 92171 & 4716 & 92870 & 1160 \\
netdiversion.mps & * & * & * & * & * & * & * \\
nexp-150-20-8-5.mps & 661 & 650 & 2707 & 5219 & 11536 & 6470 & 845 \\
ns1116954.mps & 5600 & 12858 & 441155 & 24617 & * & 81870 & 19024 \\
ns1208400.mps & 3263 & 2398 & 25447 & 25859 & 2281 & 26724 & 4545 \\
ns1644855.mps & * & * & * & * & * & * & * \\
ns1760995.mps & * & 31060 & * & * & * & * & * \\
ns1830653.mps & 18120 & 16096 & 18666 & 19139 & 4744 & 19311 & 11198 \\
ns1952667.mps & 68 & 68 & 11145 & 18491 & 87 & 24957 & 68 \\
nu25-pr12.mps & 1850 & 1642 & 6699 & 26414 & 3099 & 26281 & 6165 \\
nursesched-medium-hint03.mps & 156192 & 36390 & 160684 & 169581 & * & 168361 & 35869 \\
nursesched-sprint02.mps & 16279 & 8259 & 47487 & 46400 & 3100 & 46185 & 46907 \\
nw04.mps & 438 & 313 & 12328 & 56821 & 130 & * & 94 \\
opm2-z10-s4.mps & * & * & 28984 & * & * & * & * \\
p200x1188c.mps & 578 & 579 & 1603 & 1061 & 1177 & 1018 & 1244 \\
peg-solitaire-a3.mps & 11502 & 10055 & 31045 & 30794 & 6902 & 30915 & 19201 \\
pg.mps & 1566 & 1567 & 3628 & 3730 & 8714 & 4477 & 1552 \\
pg5\_34.mps & 1684 & 1999 & 8040 & 8664 & 3631 & 8641 & 3022 \\
physiciansched3-3.mps & * & * & * & * & * & * & * \\
physiciansched6-2.mps & * & * & * & * & * & * & * \\
piperout-08.mps & 10035 & 8071 & 77316 & 78677 & 13883 & 79658 & 10936 \\
piperout-27.mps & 13212 & 10491 & 93333 & 97329 & * & 98047 & 19730 \\
pk1.mps & 56 & 56 & 234 & 256 & 229 & 220 & 56 \\
proteindesign121hz512p9.mps & 963 & 905 & 13935 & 19537 & * & * & 1311 \\
proteindesign122trx11p8.mps & 606 & 652 & 11270 & 10568 & 2918 & 10460 & 713 \\
qap10.mps & 43368 & * & 47037 & 42295 & * & 51073 & 8507 \\
radiationm18-12-05.mps & 40126 & 50952 & * & * & * & * & * \\
radiationm40-10-02.mps & 163087 & * & * & * & * & * & * \\
rail01.mps & * & * & * & * & * & * & * \\
rail02.mps & * & * & * & * & * & * & * \\
rail507.mps & 24907 & 20067 & 195912 & 196516 & * & 197922 & 23777 \\
ran14x18-disj-8.mps & 1840 & 1657 & 3587 & 3867 & 3889 & 3821 & 2611 \\
rd-rplusc-21.mps & 475 & 412 & 342 & 913 & 640 & 840 & 517 \\
reblock115.mps & 2456 & 5185 & 11934 & 2566 & 4122 & 2627 & 3220 \\
rmatr100-p10.mps & 3409 & 4706 & 35091 & 15813 & 1672 & 17698 & 501 \\
rmatr200-p5.mps & 65929 & 85034 & * & * & * & * & 309 \\
rocI-4-11.mps & 9498 & 13144 & 46230 & 54614 & 7263 & 42979 & 20292 \\
rocII-5-11.mps & 6064 & 1315 & 3874 & 2342 & 1884 & 2382 & 1615 \\
rococoB10-011000.mps & 20776 & 1353 & 21938 & 21425 & 3977 & 21581 & 11982 \\
rococoC10-001000.mps & 3941 & 1511 & 7374 & 13936 & 5032 & 14324 & 522 \\
roi2alpha3n4.mps & 24697 & 10346 & 6264 & 32500 & 39804 & 10864 & 9391 \\
roi5alpha10n8.mps & 12459 & 6064 & 13999 & * & * & 132274 & * \\
roll3000.mps & 1815 & 2155 & 12717 & 13269 & 1892 & 13997 & 3166 \\
s100.mps & * & * & * & * & * & * & * \\
s250r10.mps & * & * & * & * & * & * & * \\
satellites2-40.mps & 73896 & 51346 & 19366 & 213784 & * & 21851 & 2498 \\
satellites2-60-fs.mps & 47644 & 41964 & 18033 & 181057 & * & 16223 & 2921 \\
savsched1.mps & * & * & * & * & * & * & * \\
sct2.mps & 6949 & 2983 & 27287 & 26895 & 7397 & 27005 & 3571 \\
seymour.mps & 7345 & 9261 & 22067 & 24986 & 2502 & 24940 & 15189 \\
seymour1.mps & 7345 & 9261 & 22067 & 24986 & 2502 & 24940 & 15189 \\
sing44.mps & 261483 & * & * & * & * & * & * \\
sing326.mps & 220981 & * & * & * & * & * & * \\
snp-02-004-104.mps & * & * & * & * & * & * & * \\
sorrell3.mps & 1397 & 1356 & 1772 & 1369 & 1335 & 1368 & 1736 \\
sp150x300d.mps & 318 & 310 & 577 & 479 & 484 & 475 & 307 \\
sp97ar.mps & 3950 & 3573 & 40951 & 50077 & 2571 & 49995 & 14935 \\
sp98ar.mps & 20342 & 10799 & 39116 & 52886 & 2519 & 52395 & 1967 \\
splice1k1.mps & 35351 & 17482 & 34404 & 48443 & * & 32872 & 16224 \\
square41.mps & * & * & * & * & * & * & * \\
square47.mps & * & * & * & * & * & * & * \\
supportcase10.mps & * & * & * & * & * & * & * \\
supportcase12.mps & * & * & * & * & * & * & * \\
supportcase18.mps & 1787 & 1077 & 22628 & 27455 & 652 & 37010 & 499 \\
supportcase19.mps & * & * & * & * & * & * & * \\
supportcase22.mps & * & 2431 & * & * & * & * & * \\
supportcase26.mps & 294 & 313 & 365 & 409 & 450 & 409 & 540 \\
supportcase33.mps & 4216 & 3176 & 44224 & 140227 & * & 87128 & 5053 \\
supportcase40.mps & 100160 & 38705 & * & 175363 & * & 181878 & 21721 \\
supportcase42.mps & * & * & * & * & * & * & * \\
supportcase6.mps & 9576 & 3895 & * & * & * & * & 931 \\
supportcase7.mps & 154110 & 24378 & * & * & * & * & * \\
swath1.mps & 640 & 1049 & 19946 & 23793 & 281 & 23712 & 479 \\
swath3.mps & 640 & 1049 & 19946 & 23793 & 281 & 23712 & 433 \\
tbfp-network.mps & 42643 & 11572 & 244825 & 237023 & * & 237077 & 3225 \\
thor50dday.mps & 38456 & * & * & * & * & * & * \\
timtab1.mps & 328 & 250 & 744 & 812 & 502 & 751 & 175 \\
tr12-30.mps & 696 & 1321 & 3304 & 5668 & 4135 & 5688 & 696 \\
traininstance2.mps & 13397 & 13713 & 21531 & 21904 & * & 32415 & 15256 \\
traininstance6.mps & 10270 & 10282 & 16929 & 15868 & 11367 & 52302 & 10512 \\
trento1.mps & 57604 & 74004 & 40969 & 40642 & * & 40722 & 34730 \\
triptim1.mps & * & * & * & * & * & * & * \\
uccase12.mps & * & * & * & * & * & * & * \\
uccase9.mps & 116722 & * & * & * & * & * & * \\
uct-subprob.mps & 9495 & 7117 & 14877 & 16206 & 2416 & 16033 & 9294 \\
unitcal\_7.mps & 150003 & 84117 & * & * & * & * & 31193 \\
var-smallemery-m6j6.mps & 59701 & 59539 & 28200 & 60575 & * & 61423 & * \\
wachplan.mps & 4655 & 5722 & 17255 & 17440 & 2080 & 16360 & 6012 \\
\end{longtable}
\restoregeometry






\end{document}